\documentclass[12pt,a4paper]{amsart}
\usepackage{geometry}
\usepackage{amsmath}
\usepackage{amssymb}
\usepackage{amsthm}
\usepackage{color}
\usepackage{xspace}
\usepackage{ulem}
\usepackage{comment}
\usepackage[dvipdfmx]{hyperref}  
\newtheorem{theorem}{Theorem}
\newtheorem{lemma}[theorem]{Lemma}
\newtheorem{corollary}[theorem]{Corollary}

\def\blem{\begin{lemma}}
\def\elem{\end{lemma}}
\def\bmat{\begin{pmatrix}}
\def\emat{\end{pmatrix}}

\def\ve{\varepsilon}
\def\ep{\varepsilon}
\def\R{\mathbb{R}}
\def\N{\mathbb{N}}
\def\ol{\overline}
\def\fr{\frac}
\def\mid{\,:\,} 
\def\gl{\lambda}
\def\IN{\text{ in }}

\def\AND{\text{ and }}
\def\FOR{\text{ for }}

\def\ON{\text{ on }}
\def\bproof{\begin{proof}}
\def\eproof{\end{proof}}
\def\stm{\setminus}

\def\pl{\partial}
\def\gd{\delta}
\def\tim{\times}
\def\bald{\begin{aligned}}
\def\eald{\end{aligned}} 
\def\FORALL{\text{ for all }}

\DeclareMathOperator\USC{USC}

\DeclareMathOperator\supp{supp}

\DeclareMathOperator\tr{tr}

\def\stm{\setminus}

\def\cS{\mathcal{S}}

\def\bye{\end{document}}
\def\by{\end{proof}\bye}
\def\fr{\frac} 
  
\def\ga{\alpha}     
\def\go{\omega}
\def\gep{\varepsilon}      
\def\ep{\gep}    
\def\mid{\,:\,}   
\def\gb{\beta} 
\def\gc{\gamma}
\def\gd{\delta}
\def\gz{\zeta} 
\def\gth{\theta}   
 
\def\gl{\lambda}
\def\gL{\Lambda}
\def\gs{\sigma}   
                  
\def\tim{\times}

\def\ol{\overline}
\def\ul{\underline}           
\def\pl{\partial}

\def\bcases{\begin{cases}}
\def\ecases{\end{cases}}
\def\balns{\begin{align*}}
\def\ealns{\end{align*}}
\def\bald{\begin{aligned}}
\def\eald{\end{aligned}}
\def\beq{\begin{equation}}
\def\eeq{\end{equation}}
\def\bred{\begin{color}{red}} \def\ered{\end{color}}
\def\gO{\Omega} 

\def\1{\mathbf{1}}

\def\IN{\text{ in }}\def\IF{\text{ if }} \def\FOR{\text{ for }} 
\def\AND{\text{ and }}
\def\RED#1{\textcolor{red}{#1}}

\def\ON{\text{ on }}

\theoremstyle{definition}
\newtheorem{definition}{Definition}
\theoremstyle{plain}
\newtheorem{proposition}[definition]{Proposition}
\theoremstyle{remark}
\newtheorem{remark}[definition]{Remark}

\def\ga{\gamma}

\def\gau{\gamma_1}

\def\gad{\gamma_2}

\def\bet{\beta}

\def\bb{m} \def\cc{n}
\allowdisplaybreaks[4]

\def\oblique{(O$_\ep$)\xspace}
\def\dirichlet{(D$_\ep$)\xspace}
\def\neumann{(N$_\ep$)\xspace}
\def\problem{(P$_\ep$)\xspace} 
\def\obliquezero{(O$_0$)\xspace}
\def\dirichletzero{(D$_0$)\xspace}
\def\neumannzero{(N$_0$)\xspace}
\def\problemzero{(P$_0$)\xspace}

\title[Thin domains with oblique boundary condition]{Fully nonlinear elliptic PDEs in thin domains with oblique boundary condition}

\author[I. Birindelli]{Isabeau Birindelli}
\address[I. Birindelli]{Dipartimento di Matematica Guido Castelnuovo, Sapienza 
Universit\`a di Roma, Piazzale Aldo Moro 5, Roma, Italy.}
\email{isabeau@mat.uniroma1.it}

\author[A. Briani]{Ariela Briani}
\address[A. Briani]{Institut Denis Poisson, 
Universit\'e de Tours, France.}
\email{ariela.briani@univ-tours.fr}

\author[H. Ishii]{Hitoshi Ishii
}
\address[H. Ishii]{Institute for Mathematics and Computer Science, Tsuda  University,
 2-1-1 Tsuda, Kodaira, Tokyo 187-8577 Japan.}
\email{hitoshi.ishii@waseda.jp}


\thanks{A. Briani was partially supported by l’Agence Nationale de la Recherche (ANR), project
ANR-22-CE40-0010 COSS; I. Birindelli was partially supported by project Leoni 2023 GNAMPA-INDAM and project  "At the Edge of Reaction-diffusion equations" Sapienza Università di Roma. H. {Ishii was partially supported by the JSPS KAKENHI Grant Nos. JP20K03688, JP23K20224, and JP23K20604.}} 


\begin{document}

\begin{abstract}
This note is the natural continuation of what was started in \cite{BBI} i.e. the 
extension to fully nonlinear operators of  the well known result on thin domains of  Hale and Raugel \cite{HR}. Here we consider oblique 
boundary condition,  and find some new phenomena, in particular the limit equations contain "new terms"  in the second, first and zeroth 
order terms which don't  have an equivalent in the Neumann case. 
\end{abstract}

\maketitle 

\tableofcontents
\allowdisplaybreaks

\def\L{\mathrm{L}} \def\T{\mathrm{T}}\def\B{\mathrm{B}}
\section{Introduction} 
Let $\gO_\ep$ be a cylindrical domain, which is thin in the direction of the axe i.e.
$$\gO_\ep=\{(x,y)\in\gO\tim\R\mid \ep g^-(x)<y<\ep g^+(x)\}$$
with  $\gO\subset\R^N$ a 
bounded open domain, $g^\pm\in C(\ol\gO,\R)$, and $g^-(x)<g^+(x)$. 
We think of  $\pl\gO_\ep$, as divided in three parts the top, the bottom and the lateral part:
\[\begin{gathered}
\pl_\T \gO_\ep=\{(x,y)\mid  x\in \ol\gO, y=\ep g^+(x)\},\qquad
\pl_\B\gO_\ep =\{(x,y)\mid x\in\ol\gO, y=\ep g^-(x)\}, \quad
\\ \pl_\L\gO_\ep=\{(x,y)\mid x\in\pl\gO, \ep g^-(x)\leq y\leq \ep g^+(x)\}. 
\end{gathered}\]

We consider the elliptic PDE problem with oblique boundary conditions on the top and bottom of the domain:
\beq\label{eq1.1}
\left\{
\bald
F(D^2u^\ep,Du^\ep,u^\ep,x,y)=0 \ \IN\gO_\ep, &\\
\gamma^+\cdot Du^\ep=\beta^+ \ \ \ON\pl_\T\gO_\ep,  &\ \ \ \gamma^-\cdot Du^\ep=\beta^- \ \ \ON \pl_\B\gO_\ep,\\
B(Du^\ep,u^\ep,x)=0 \ON\pl_\L\gO_\ep,
\eald
\right.
\eeq
where $\gamma^\pm\in C(\ol\gO\tim[-1,1],\R^{N+1})$ and $\beta^\pm \in C(\ol\gO\tim[-1,1],\R)$. 

Our purpose in this paper is to investigate the asymptotic behaviour of $u^\ep$ when the positive parameter $\ep$ is sent to zero. In particular we want to determine the equation satisfied by the limit function in $\Omega$. This paper is the natural continuation of \cite{BBI}. In that work we had generalized to a non variational setting the famous result of Hale and Raugel \cite{HR} using the so called test function approach \`a la Evans. We shall see here that, choosing the oblique boundary condition, completely changes the limit equation even in the case of the Laplacian, introducing new first and second order terms.

On the other hand, the choice of the boundary condition on the lateral boundary 
$\pl_\L\gO_\ep$ does not play any crucial role. Hence we shall prove the results with either Neumann, or oblique or Dirichlet boundary condition on the lateral boundary i.e. either $B(Du^\ep,u^\ep,x):=\gamma\cdot Du^\ep-\beta(x)$ or $B(u^\ep):=u^\ep-\beta(x)$. For the clarity of the discussion, in the introduction and in the formal expansion of the next section, we concentrate mainly on the standard Neumann condition:
\beq\label{eq1.3}
B(Du^\ep,u^\ep,x):=\nu \cdot D_xu^\ep =0 \ \ \ON \pl_\L\gO_\ep,
\eeq 
where $\nu=\nu(x)$ denotes the outward unit normal vector at $x\in \pl\gO$. Notice that
$\nu_\ep(x,y)=(\nu(x),0)\in\R^{N+1}$ is the outward unit  normal vector at $(x,y)\in\pl_\L\gO_\ep$.

Regarding the obliqueness of $\gamma^\pm$, we assume that, when we write $\gamma^\pm=(\gamma_1^\pm,\gamma_2^\pm)$, with 
$\gamma_1^\pm\in C(\ol\gO\tim[-1,1],\R^N)$ and $\gamma_2^\pm\in C(\ol\gO\tim[-1,1],\R)$,  
\beq\label{eq1.2+1}
\pm\gamma_2^\pm >0 \ \ \ON \ol\gO\tim[-1,1]. 
\eeq

Hence, without loss of generality, we can suppose that 
\beq\label{eq1.2+2}
\gamma_2^+=1\quad\AND\quad \gamma_2^-=-1.
\eeq
So we will consider the oblique boundary condition
\beq\label{eq1.2}
\gamma_1^+\cdot D_xu^\ep+u^\ep_y=\beta^+ \ \ \ON\pl_\T\gO_\ep 
\quad \AND\quad
\gamma_1^-\cdot D_xu^\ep-u^\ep_y=\beta^- \ \ \ON \pl_\B\gO_\ep.
\eeq 
If $g^\pm\in C^1(\ol\gO,\R)$, $\nu_\ep$ denotes the outward unit normal vector 
on $\pl_\T\gO_\ep\cup \pl_\B\gO_\ep$, and $\ep>0$ is sufficiently small, then $
\nu_\ep\cdot (\gamma_1^+,1)>0$ on $\pl_\T\gO_\ep$ 
and $\nu_\ep\cdot (\gamma_1^-,-1)>0$ on $\pl_\B\gO_\ep$.

To present the results, we need some hypotheses and notation. As basic hypotheses, we assume 
throughout this paper the following (H1) and (H2):
\begin{enumerate}
\item[(H1)] $F\in C(\cS(N+1)\tim \R^{N+1}\tim\R\tim \ol\gO\tim[-1,1],\R)$, and  for every $(p,x,y)\in\R^{N+1}\tim\ol\gO\tim[-1,1]$ and some constant $\alpha>0$, if $X, Y\in\cS(N+1)$, $r,s\in\R$, $X\geq Y$ in the semi-definite sense, and $r\leq s$ then  $F(X,p,r,x,y)-\alpha r\leq F(Y,p,s,x,y)-\alpha s$.  
\item[(H2)] The set $\gO$ is a bounded, open domain with $C^1$ boundary,  
$\gamma^\pm\in C(\ol\gO\tim[-1,1],\R^{N+1})$, and $\beta^\pm \in C(\ol\gO\tim[-1,1],\R)$.
\end{enumerate}
Here and below, $\cS(k)$ denotes the space of $k\tim k$
real symmetric matrices. 

We always assume that
\beq\label{eq1.3+0}    
g^\pm\in C^1(\ol\gO),\quad g^-(x)<g^+(x) \ \ \FOR x\in\ol\gO.  
\eeq
Since we are interested in the asymptotic behaviour, as $\ep\to 0^+$, of the solutions to the boundary value problem \eqref{eq1.1},  we may focus our 
considerations only on small $\ep$, say $0<\ep<\ep_0$, which allows us to assume always that  
\beq \label{eq1.3+1}\left\{\,\bald
&\ol{\gO_\ep}\subset \ol\gO\tim[-1,1],
\\&\nu_\ep\cdot (\gamma_1^+,1)> 0 \ \ \ON \pl_\T\gO_\ep,
\\&\nu_\ep\cdot (\gamma_1^-,-1)>0 \ \ \ON \pl_\B\gO_\ep. 
\eald\right.
\eeq

A crucial assumption on $\gamma^\pm$ and $\beta^\pm$ is the following:
\beq \label{eq1.4}
\beta^+(x,0)=-\beta^-(x,0)=:\beta_o(x)\ \ \AND \ \ 
\gamma_1^+(x,0)=-\gamma_1^-(x,0)=:\gamma_o(x) \ \FOR x\in\ol\gO. 
\eeq

The necessity of this assumption will be evident in Section 2 where we perform a formal expansion for the clarity of the argument.

We assume that 
\beq\label{eq1.5}
\beta_o\in C^1(\ol\gO,\R) \ \ \AND \ \ \gamma_o\in C^1(\ol\gO,\R^N). 
\eeq
In this paper, our notation that $f\in C^i(K,\R^k)$ for a compact subset $K$ of $\R^j$ indicates 
that $f\mid K\to \R^k$ is the restriction of a function $g\in C^i(\R^j,\R^k)$ to the set $K$.  

We need to assume that for some $k^\pm\in C(\ol\gO,\R^N)$ and $l^\pm\in C(\ol\gO,\R)$, 
\beq \label{eq1.6} \left\{\,\bald
&\gamma_1^\pm(x,y)=\pm\gamma_o(x)+k^\pm(x)y+o(|y|)
\\&\beta^\pm(x,y)=\pm\beta_o(x)+l^\pm(x)y+o(|y|),
\eald \right.
\eeq
as $y\to 0$, where $o(|y|)/|y|\to 0$ uniformly on $\ol\gO$ as $y\to 0$.

Having introduced all the terms in the first order expansion in $y$, we are in a position to write down the limit equation.

To simplify the notations we introduce the terms $b\in C(\ol\gO,\R^N),\,c\in C(\ol\gO,\R)$ :
\begin{align}
\label{eq1.7}b(x)&= \gamma_o(D\gamma_o)^\T
-\fr{1}{g^+-g^-}
\big( g^+k^++g^-k^-
\big),
\\ \label{eq1.8}
c(x)&= -\gamma_o\cdot D\beta_o+
\fr{1}{g^+-g^-}  
\big(g^+l^++g^-l^-
\big),                           
\end{align}
where $D\gamma_o(x)$ denotes the matrix, whose $(i,j)^{\,\text{th}}$ entry is given by $\pl{\gamma_o}_i(x)/\pl x_j$ 
and $D\beta_o(x)$ denotes the row vector, whose $i^{\,\text{th}}$ entry is $\pl \beta_o(x)/\pl x_i$. 
From now on, as in \eqref{eq1.7} and \eqref{eq1.8}, in order to simplify notation, when a function $f$ is given and no confusion is expected, the value $f(x)$ of the function $f$ at $x$ is often written as $f$.  

We define $G\mid \cS(N)\tim\R^N\tim\R\tim \ol\gO\to\R$ by
\beq\label{eq1.8+1}
G(X,p,r,x)=F(A+B+C, (p,\beta_o-\gamma_o\cdot p), r,(x,0)),
\eeq
where 
\beq\label{eq1.8+2}\bald
A&=\begin{pmatrix}X &-X\gamma_o(x)^\T\\ -\gamma_o(x)X&\gamma_o(x)X\gamma_o(x)^\T\end{pmatrix},
\\ B&=\begin{pmatrix}0 &-(pD\gamma_o(x))^\T\\ -pD\gamma_o(x) & b(x)\cdot p\end{pmatrix},
\\ C&=\begin{pmatrix}0 &D\beta_o(x)^\T\\ D\beta_o(x) & c(x)\end{pmatrix}.
\eald
\eeq
We remark that if $A(X)$ denotes the above matrix $A$ and $X\geq 0$, then 
\[
A(X)=\begin{pmatrix}I_N \\ -\gamma_o(x)\end{pmatrix}X\begin{pmatrix}I_N &-\gamma_o(x)^\T\end{pmatrix}\geq 0,
\]  
which ensures that $G$ is degenerate elliptic, i.e., if $X,Y\in\cS(N)$ and $X\leq Y$, then 
$G(X,p,r,x)\geq G(Y,p,r,x)$ for every $(p,r,x)\in\R^N\tim\R\tim\ol\gO$.  In the above and below, 
$I_N$ denotes the $N\tim N$ identity matrix. 

The limit problem for \eqref{eq1.1} with \eqref{eq1.3}, is the boundary value problem
\begin{align}
\label{eq1.9}G(D^2u,Du,u,x)&=0 \ \ \IN\gO,
\\ \label{eq1.10}\nu\cdot Du&=0 \ \ \ON\pl\gO.
\end{align}
When the lateral condition is either oblique or Dirichlet then the limit equation inherits on $\partial\Omega$ the lateral condition.

In the basic case $F(D^2u)=-\Delta u$, we get that the limit equation is given by
 \beq {G(D^2u,Du,u,x)= -(\Delta u + \gamma_o(x)D^2u\gamma_o(x)^\T+b(x)\cdot Du+ c(x)u). }\eeq
 
When the boundary data on the top and bottom is not oblique but it is just the Neumann boundary condition, and when $g^-(x)=0$ as in our previous paper
 \cite{BBI}, the previous quantities $\gamma_o$, $\beta_o$, $k^-$ and $l^\pm$ are zero 
while $k^+(x)=-\frac{D g^+}{g^+}$ and hence we recover the limit equation:
 $$F\Big(\begin{pmatrix}X &0\\ 0& {\frac{Dg^+\cdot Du}{g^+(x)}}\end{pmatrix},(Du, 0), u , (x,0)\Big)=0 \ \IN\gO, \ \ \ \nu\cdot Du=0 \ \ \ON\pl\gO.$$
 
 As discussed also in  \cite{BBI},  the presence of corners in  $\gO_\ep$ requires a little care in the  definition of sub and super  viscosity solution.
Hence we give a definition of viscosity solution to \eqref{eq1.1}, which, 
contrary to \cite{CIL}, does not require the continuity of the solution. 
The precise  definition of sub and super viscosity solution  up to the  boundary  is given in Section 3 below but 
let us anticipate that we call a bounded function $u$ on 
$\ol{\gO_\ep}$ a viscosity solution to \eqref{eq1.1} if 
the upper and lower semicontinuous envelopes of $u$ are viscosity sub and super solutions  
to \eqref{eq1.1} in the sense of  \cite{CIL}, respectively. 
The definition of viscosity solution to \eqref{eq1.9}--\eqref{eq1.10} is made 
in the same way. 

The first result we want to quote is an existence result for $\ep$ fixed small. We shall state it when, on the lateral boundary, the Neumann condition holds:
 \begin{proposition}\label{prop1} Assume (H1), (H2), 
\eqref{eq1.2+2}, 
\eqref{eq1.3+0}, \eqref{eq1.4}, 
\eqref{eq1.5}, and \eqref{eq1.6}. 
Then, 
there exist positive constants $\ep_1<\ep_0$ and $C_0$ such that for each $0<\ep<\ep_1$, there is a viscosity solution  $u^\ep$ to \eqref{eq1.1} with  \eqref{eq1.3}  furthermore any solution $u^\ep$ will satisfy 
$$\sup_{\ol{\gO_\ep}}|u^\ep|\leq C_0.$$ 
\end{proposition}

We can now state the convergence theorem.
Let $S_\ep$ denote the set of viscosity solutions to \eqref{eq1.1} with \eqref{eq1.3} for given 
$\ep>0$.
Under the hypotheses of Proposition \ref{prop1}, we define the half relaxed limits $u^\pm$ 
by 
\[\bald
u^+(x)&=\lim_{r\to 0^+}\sup\{u(\xi,\eta)\mid u\in S_\ep,\ (\xi,\eta)\in \ol{\gO_\ep},\ 
0<\ep<r,\ |\xi-x|<r \},
\\ u^-(x)&=\lim_{r\to 0^+}\inf\{u(\xi,\eta)\mid u\in S_\ep,\ (\xi,\eta)\in\ol{\gO_\ep},\ 
0<\ep<r,\ |\xi-x|<r\},
\eald
\]
which are bounded functions on $\ol\gO$ by Proposition \ref{prop1}.

\begin{theorem}\label{thm1} Under the hypotheses of Proposition \ref{prop1}, 
the functions $u^+$ and $u^-$ are a viscosity sub and super solutions to \eqref{eq1.9}--\eqref{eq1.10}, respectively. 
\end{theorem}

Clearly, if $G$, as defined in \eqref{eq1.8+1} and \eqref{eq1.8+2}, satisfies the comparison principle then 
$u_o:=u^+=u^-$ and it is a solution of \eqref{eq1.9}--\eqref{eq1.10}. Furthermore $u^\ep$ converges uniformly to $u_o$, see Corollary \ref{cor1}.

\medskip

Remark that the assumption \eqref{eq1.4} is indeed crucial. Consider problem:
$$\left\{
\begin{array}{l}
-u_{yy}+u=1 \ \ \ \  \IN (0,1)\times (0,\ep),\\
u_y(x,\ep)=1\ ,  \ \ -u_y(x,0)=0 \ ,  \ \ 
u_x(0,y)=0,\ \ \  u_x(1,y)=0
\end{array}
\right.
$$
for which all hypotheses, but \eqref{eq1.4}, of Proposition \ref{prop1} are satisfied. Clearly its solution is  $u^\ep(x,y)=\frac{1}{e^\ep-e^{-\ep}}\left(e^y+e^{-y}\right)+1$, which is not bounded as $\ep$ goes to zero.

As already mentioned, in the paper the existence results  and the convergence results will be stated and proved for a more general case  of lateral conditions, see Proposition \ref{exisnew} and Theorem \ref{thm1new}.

After the work of Hale and Raugel, there has been a number of interesting results that generalise it in many {contexts} see e.g. the works of Arrieta, Nakasato, Nogueira, Perreira, Villanueva-Pesqueira \cite{AP,ANP, 
ANaPe, AVP} .... with for example oscillating boundary or quasi linear operators.
On the other hand there is also a very large number of papers that treat the “variational” convergence of  functionals $F_\ep$ to a new functional defined in a lower dimensional space, where the main tool is $\Gamma$ convergence, it is impossible to mention all the works in that realm, but we can mention e.g. \cite{ABuPe,
AnBaPe, BFF}. But, even if the problems are not so far mathematically, we believe that the point of view are so different that it would require quite some work to make the connection. And even if this would be very interesting it is certainly behind the scope of this introduction.

The paper is organised in the following way: In the next section we perform the formal expansion that will allow us to guess the limit equation and upload the test function approach. In the third section we prove the existence and bounds at $\ep$ fixed. In the fourth section we prove the main asymptotic results and finally in the last section we prove that when $F$ is a Bellman-Isaacs operator under the right general condition the limit operator satisfies the comparison principle i.e. the convergence is uniform.

\section{Formal expansion} To gain an intuitive understanding of \eqref{eq1.9}, 
we perform the following formal calculations. First, we assume that the following 
expansion, as $\ep\to 0^+$, is valid in an appropriate sense: 
\[\bald
&u^\ep(x,y)=u^0(x)+\ep u^1(x,y/\ep)+\ep^2u^2(x,y/\ep)+\cdots,
\eald
\]
where $u^0\mid \ol\gO\to \R$, $u^1, u^2\mid \ol\gO\tim[-1,1]\to \R$, and 
every ingredients appearing above and elsewhere  are supposed to be smooth enough.  
At this point, we do not require condition \eqref{eq1.4} yet, but consider the expansion formula \eqref{eq1.6}. That is, we set it as follows. 
\beq \label{eq1.6+1} \left\{\,\bald
&\gamma_1^\pm(x,y)=\gamma_1^\pm(x,0)+k^\pm(x)y+o(|y|)
\\&\beta^\pm(x,y)=\beta^\pm(x,0)+l^\pm(x)y+o(|y|).
\eald \right.
\eeq
Insert the above expansions 
into the boundary condition \eqref{eq1.2}, on the respective boundaries $\pl_\T\gO_\ep$ and $\pl_\B\gO_\ep$, 
\[\bald
0&
=\gamma_1^\pm\cdot D_xu^\ep\pm u^\ep_y-\beta^\pm
\\&=\gamma_1^\pm\cdot D u^0 +\ep \gamma_1^\pm\cdot D_xu^1
+\ep^2 \gamma^\pm\cdot D_xu^2 
 \pm u^1_y
\pm \ep u^2_y -\beta^\pm  +\cdots 
\\& =
\gamma_1^{\pm}(x,0)\cdot Du^0(x)+\ep g^\pm k^\pm(x) \cdot Du^0(x)
+\ep \gamma_1^{\pm}(x,0)\cdot D_xu^1(x,g^\pm) 
\\&\quad \pm u^1_y(x,g^\pm)
\pm \ep u^2_y(x,g^\pm)  -\beta^{\pm}-\ep g^\pm l^\pm(x)+\cdots.
\eald
\]
Here and below, for function $u=u(x,y)$, $u_y$ denotes the $y$-derivative $\pl u/\pl y
=D_y u$ of $u$. 
Equating to zero the coefficients of $\ep^0$ and $\ep^1$ on the right-hand side of the above, 
we obtain 
\begin{align}
& 0=\gamma_1^{\pm}(x,0)\cdot Du^0(x)\pm u^1_y(x,g^\pm(x))
- \beta^{\pm}(x,0),
\label{eq2.1}
\\&0=g^\pm k^\pm(x)\cdot Du^0(x)+\gamma_1^{\pm}(x,0)\cdot D_xu^1(x,g^\pm(x))
 \label{eq2.2}
\\&\quad \pm u^2_y(x,g^\pm(x))
-g^\pm(x)l^\pm(x).\notag
\end{align}
In the domain $\gO_\ep$, we have
\beq\label{eq2.3}
0=F(D^2u^\ep,Du^\ep,u^\ep,x,y),
\eeq 
with 
\beq \label{eq2.4} \left\{
\bald
D^2u^\ep&= \begin{pmatrix} D_x^2u^0&(D_x u_y^1)^\T\\ D_xu^1_y& \ep^{-1}u^1_{yy}+u^2_{yy} \end{pmatrix}+o(1), 
\\ Du^\ep&=(D_xu^0,u^1_y)+o(1),\qquad 
u^\ep=u^0(x)+o(1).
\eald \right.
\eeq
A natural ansatz here, to obtain a PDE for $u^0=\lim_{\ep\to 0^+}u^\ep$, is to impose that 
\[
u^1_{yy}(x,y)=0 \ \ \FOR x\in\ol\gO, \, g^-(x)< y<g^+(x).
\]
To achieve this, we assume that there is a function $v\mid\ol\gO\to \R$ such that
\beq\label{eq2.5}
u^1(x,y)=v(x)y \ \ \FOR y\in\R.
\eeq
With this ansatz, \eqref{eq2.1} and \eqref{eq2.2} read 
\begin{align}
& 0=\gamma_1^{\pm}(x,0)\cdot Du^0(x)\pm v(x)-
\beta^{\pm}(x,0), 
\label{eq2.6}
\\
&0=g^\pm k^\pm(x)\cdot Du^0(x)+
 g^\pm\gamma_1^{\pm}(x,0)\cdot Dv(x)
 \pm u_y^2(x,g^\pm)-g^\pm l^\pm(x). \label{eq2.7}
\end{align} 
The solution of \eqref{eq2.6} with respect to $v$ is 
$$
v(x)=\beta^+(x,0)-\gamma_1^+\cdot Du^0(x) =-\beta^-(x,0)+\gamma_1^-\cdot Du^0(x).
$$
This is well defined if condition \eqref{eq1.4} holds, and it determines the value for $v$ as 
\beq\label{eq2.5+1} v(x):=\beta_o-\gamma_o\cdot Du^0(x),
\eeq
and the expansion formula \eqref{eq1.6+1} turns out to be identical to \eqref{eq1.6}.

\noindent With this choice of $v$, from  \eqref{eq2.7}, we obtain
\beq \label{eq2.8+3}
u^2_y(x,g^\pm(x))=\pm g^\pm(l^\pm(x)
-k^\pm(x)\cdot Du^0(x)\mp\gamma_o(x)\cdot Dv  ). 
\eeq 

\noindent For each $x$, we look for a quadratic function $Q(y)$ with slope 
$u^2_y(x,g^-(x))$ and $u^2_y(x,g^+(x))$ at $y=g^-(x)$ and $y=g^+(x)$, respectively, 
and choose: 
\[
Q(y)=\fr{(y-g^-)^2}{2(g^+-g^-)}u^2_y(x,g^+)-
\fr{(y-g^+)^2}{2(g^+-g^-)}u^2_y(x,g^-). 
\]
Thus, our choice of  $u^2$ is: 
\beq \label{eq2.8+4}\bald
u^2(x,y)&=\fr{(y-g^-)^2g^+}{2(g^+-g^-)}
\Big(l^+(x)
-k^+(x)\cdot Du^0(x)-\gamma_o\cdot Dv
\Big)
\\&\quad+\fr{(y-g^+)^2g^-}{2(g^+-g^-)}\Big(l^-(x)
-k^-(x)\cdot Du^0(x)
+\gamma_o\cdot Dv \Big).
\eald
\eeq
We set 
\beq \label{eq2.8+5}
w^\pm(x)= \fr{g^\pm}{(g^+-g^-)}(l^\pm(x)
-k^\pm(x)\cdot Du^0(x)\mp\gamma_o\cdot Dv ), 
\eeq
so that \eqref{eq2.8+4} can be written as
\beq\label{a1}
u^2(x,y)=\fr 12(y-g^-(x))^2w^+(x)+\fr 12(y-g^+(x))^2w^-(x). 
\eeq

In view of \eqref{eq2.4}, we wish to give convenient formulas for $Du^1$ and $u^2_{yy}$, 
which involve explicitly neither functions $u^1$ nor $u^2$.  We compute that
\begin{align} \label{a2}
Du^1_y&=Dv=D(\beta_o-\gamma_o\cdot Du^0(x) )
\\\notag &=D\beta_o-Du^0
D\gamma_o-\gamma_oD^2u^0,
\end{align}
\begin{align}\label{a3}
\gamma_o \cdot Dv
&=\gamma_o \cdot D(\beta_o-\gamma_o\cdot Du^0) 
\\ \notag&=\gamma_o\cdot D\beta_o-\gamma_o\cdot Du^0 D\gamma_o-\gamma_oD^2u^0 \gamma_o^\T
\\ \notag&=\gamma_o\cdot D\beta_o-[\gamma_o(D\gamma_o)^\T]\cdot Du^0-\gamma_oD^2u^0 \gamma_o^\T,
\end{align}
and 
\begin{align}\label{a4}
u^2_{yy}(x,y)&=w^+(x)+w^-(x)
\\\notag&= \fr{1}{g^+-g^-}\Big(g^+l^++g^-l^-\Big)
-\fr{1}{g^+-g^-}\Big(g^+k^++g^-k^-\Big)\cdot Du^0
 \quad -\gamma_o\cdot Dv
\\\notag&
= \fr{1}{g^+-g^-}\Big(g^+l^++g^-l^-\Big)
-\fr{1}{g^+-g^-}\Big(g^+k^++g^-k^-\Big)\cdot Du^0
\\\notag& \quad -\gamma_o\cdot D\beta_o+[\gamma_o(D\gamma_o)^\T]\cdot Du^0+\gamma_oD^2u^0\,\gamma_o^\T.
\end{align}
Recalling \eqref{eq1.7} and \eqref{eq1.8}, we obtain
\beq\label{a5}
u^2_{yy}=\gamma_o D^2u^0 \gamma_o^\T+b\cdot Du^0+c.  
\eeq
We go back to \eqref{eq2.4}, to find that as $\ep\to 0^+$,
\[\bald
D^2u^\ep&=\begin{pmatrix} D^2u^0& (D\beta_o-Du^0D\gamma_o-\gamma_oD^2u^0)^\T\\ 
D\beta_o-Du^0D\gamma_o-\gamma_oD^2u^0& \gamma_o D^2u^0 \gamma_o^\T+b\cdot Du^0+c
\end{pmatrix}+o(1)
\\&=\begin{pmatrix} D^2u^0&-D^2u^0 \gamma_o^\T\\ 
-\gamma_oD^2u^0& \gamma_o D^2u^0 \gamma_o^\T
\end{pmatrix}
+\begin{pmatrix} 0&-(Du^0D\gamma_o)^\T\\ 
-Du^0D\gamma_o& b\cdot Du^0
\end{pmatrix}
+\begin{pmatrix} 0& D\beta_o^\T\\ 
D\beta_o& c
\end{pmatrix}
+o(1),
\eald
\]
and
\[
Du^\ep=(Du^0,v)+o(1)=(Du^0,\beta_o-\gamma_o\cdot Du^0)+o(1). 
\]
Accordingly, in the limit as $\ep\to 0^+$, we deduce that \eqref{eq1.9} is the equation for the limit function $u^0$ of the solutions $u^\ep$ to \eqref{eq1.1}
 and 
\eqref{eq1.2}.

\section{Existence and bounds}

As mentioned in the introduction and as seen in  \cite{BBI},  the presence of corners in  $\gO_\ep$ requires a little care in the  definition of viscosity sub and super solution.  Instead since the choice of lateral condition does not affect the proofs of the main results we now introduce the more general case.

But first some notations:
In the following proofs, we adopt the $o(\ep^p)$ and $O(\ep^p)$ symbols, 
to indicate, respectively, functions $q(x,\ep)$ and $r(x,\ep)$ on $\ol\gO$ 
such that $|q(x,\ep)|/\ep^p$ converges uniformly to $0$ on $\ol\gO$ 
and $r(x,\ep)/\ep^p$ is uniformly bounded on $\ol\gO$ as $\ep\to 0^+$. 
Similarly, when $q(x,y,\ep)$ and $r(x,y,\ep)$ are functions of $\ol{\gO_\ep}$,
$q=o(\ep^p)$ and $r=O(\ep)$ mean that, as $\ep\to 0^+$, $\sup_{\ol{\gO_\ep}}|q(x,y,\ep)|/\ep^p\to 0$ and $\sup_{\ol{\gO_\ep}}|r(x,y,\ep)/\ep^p|$ is bounded by  a constant.  

Recall that $\ep_0>0$ has been chosen so that \eqref{eq1.3+1} holds, that is, for $\ep\in(0,\ep_0)$,   
we have 
\begin{gather}\notag  \ol{\gO_\ep}\subset \ol\gO\tim[-1,1],
\\ \label{oblique_TB}
-\ep Dg^+\cdot \gamma_1^++1>0 \ \ \ON \pl_\T\gO_\ep \ \ \ \AND
\ \ \ \ep Dg^-\cdot \gamma_1^-+1>0 \ \ \ON \pl_\B\gO_\ep.
\end{gather}

Let $\beta\in C(\pl\gO\tim[-1,1],\R)$ and $\gamma\in C(\pl\gO\tim[-1,1],\R^{N+1})$ be given.  We are interested in the boundary value problems 
for \eqref{eq1.1}, with the boundary condition 
\eqref{eq1.2} on the top and bottom of the boundary and either of the conditions on the lateral boundary $\pl_\L\gO_\ep$: 
\beq \label{eq.lbc.2}
 \gamma\cdot Du=\beta, 
\eeq
or 
\beq \label{eq.lbc.3}  
u=\beta.
\eeq
In the case of \eqref{eq.lbc.2} we need {the} following conditions
 \beq \label{eq5.2111}\left\{\,\bald 
&\gamma=(\gamma_1,\gamma_2)\in C(\pl\gO\tim[-1,1],\R^{N+1}), \quad\beta\in C(\pl\gO\tim[-1,1],\R), 
\\& \gamma_1\cdot\nu>0 \ \ \  \ON \pl\gO\tim[-1,1],\quad \gamma_2(x,0)=0 \ \ \ \FOR x\in \pl\gO. 
\eald
\right. 
\eeq
For notational convenience, given $\ep>0$ we call 
problem \oblique (resp., problem \dirichlet) 
problem \eqref{eq1.1}, 
with boundary condition $B$ as in \eqref{eq.lbc.2} 
(resp., problem \eqref{eq1.1},
with boundary condition $B$ as in \eqref{eq.lbc.3}).

Notice that problem 
\eqref{eq1.1}--\eqref{eq1.3} is a particular case of problem \oblique, with $\gamma=\nu_\L$
and $\beta=0$,  which we call \neumann.  

In the sequel when 
we focus on either problem \neumann, \oblique, or \dirichlet we call it \problem.

Viscosity solutions to these problems are defined as follows. 
When $u$ is a bounded 
function on $\ol{\gO_\ep}$, we call $u$ a viscosity subsolution of problem \oblique \ (resp., problem \dirichlet )
if the following condition holds 
for the upper semicontinuous envelope $v$ 
of $u$: 
whenever $\phi\in C^2(\ol{\gO_\ep})$, $\hat z=(\hat x,\hat y)\in\ol{\gO_\ep}$ and 
$\max_{\ol{\gO_\ep}}(v-\phi)=(v-\phi)(\hat z)$, we have 
\beq\label{sub1}
F(D^2\phi(\hat z),D\phi(\hat z),v(\hat z),\hat z)\leq 0
\eeq
if $\hat z\in\gO_\ep$,  we have either \eqref{sub1} or 
\beq \label{subL}
\gamma (\hat z)\cdot D\phi(\hat z)\leq \beta(\hat z)  \  \ (\mbox{resp., } v(\hat z)\leq \beta(\hat z)), 
\eeq
if $\hat z\in\pl_\mathrm{L}\gO_{\ep}\stm(\pl_\mathrm{B}\gO_\ep \cup\pl_\mathrm{T}\gO_\ep)$, 
we have either \eqref{sub1} or 
\beq \label{subB}
 \gamma^-  (\hat z) \cdot D\phi(\hat z)  \leq \beta^- (\hat z)
\eeq if  $\hat z\in\pl_\mathrm{B}\gO_{\ep}\stm\pl_\mathrm{L}\gO_\ep$, 
we have either \eqref{sub1} or
\beq
\label{subT}
\gamma^+ (\hat z)\cdot D\phi(\hat z) \leq \beta^+ (\hat z)
\eeq
if $\hat z\in\pl_\mathrm{T}\gO_\ep\stm\pl_\mathrm{L}\gO_\ep$,
we have either \eqref{sub1}, \eqref{subL}, or \eqref{subB} if $\hat z\in\pl_\mathrm{L}\gO_\ep\cap\pl_\mathrm{B}\gO_\ep$, and we have either \eqref{sub1}, \eqref{subL}, or 
 \eqref{subT} if $\hat z\in \pl_\mathrm{L}\gO_\ep\cap\pl_\mathrm{T}\gO_\ep$.
Replacing ``$\max$'', ``$\leq$'', and ``upper semicontinuous envelope", with ``$\min$'', ``$\geq $'', 
and ``lower semicontinuous envelope'' respectively, in the above condition yields the right definition of viscosity supersolution. 
Viscosity solutions are functions which are both viscosity sub and super solutions. 

The main result of this section is the following:

\begin{proposition}\label{exisnew}
 Assume (H1), (H2), 
\eqref{eq1.2+2},
\eqref{eq1.3+0}, \eqref{eq1.4}, 
\eqref{eq1.5}, and \eqref{eq1.6}. 
Then, 
there exist positive constants $\ep_1<\ep_0$ and $C_0$ such that for each $0<\ep<\ep_1$, there is a viscosity solution  $u^\ep$ to :
$$\mbox{\eqref{eq1.1}, \eqref{eq1.2}, and either}\ \   \eqref{eq.lbc.2} \mbox{ with \eqref{eq5.2111} or \eqref{eq.lbc.3} }, \mbox{on} \ \ \partial_L\Omega_\varepsilon.$$
 Furthermore any solution $u^\ep$ will satisfy 
$\sup_{\ol{\gO_\ep}}|u^\ep|\leq C_0$. 
\end{proposition}
Before starting the proof we need a few technical results.

\begin{lemma}\label{lem1} Let $u\in C(\ol\gO,\R)$.
There is a family $\{u_\ep \mid \ep>0\}$ 
of functions in $C^\infty(\ol\gO,\R)$ such that as $\ep \to 0^+$, 
\[
u_\ep=u+o(1) \ \ \AND \ \ \ Du^\ep= o(\ep^{-1}). 
\]
\end{lemma} 

\bproof By extending the domain of definition of $u$, we may assume that $u$ is defined,
bounded, uniformly continuous in $\R^N$. Let $\gl \in C^\infty(\R^N,\R)$ be a standard mollification 
kernel, with $\supp \gl$ being contained in the unit ball $B_1$ centered at the origin.  Let 
$\gl_\ep$ be defined by $\gl_\ep(x)=\ep^{-N}\gl (x/\ep)$. Let $m$ be the modulus of continuity 
of $u$. Set $u_\ep=u*\gl_\ep$ for $\ep>0$, and observe that
\[\bald
|u_\ep(x)-u(x)|&\leq\int_{B_\ep}\gl_\ep(y)|u(x-y)-u(x)|dy\leq m(\ep)\int_{B_\ep}\gl_\ep(y)dy=m(\ep),
\\ |Du_\ep(x)|&=\Big|\int_{B_\ep}u(x-y)D\gl_\ep(y)dy\Big|=\Big|\int_{B_\ep}(u(x-y)-u(x))D\gl_\ep(y)dy\Big|
\\&
\leq m(\ep)\int_{B_\ep}|D\gl(y/\ep)| \ep^{-N-1}dy=m(\ep)e^{-1}\int_{B_1}|D\gl(y)|dy. s
\eald
\]
\eproof

\begin{remark}\label{rem1} The above proof extends to a general case 
where $u\in C^k(\ol\gO,\R)$,
with nonnegative integer $k$. Let $u_\ep$ be the function defined as above. Fix  any $j\in\N$.  
The conclusion is: for multi-index $\gs=(\gs_1,\ldots,\gs_N)$, if $|\gs|< k$, then 
\[
D^\gs u_\ep=D^\gs u+O(\ep),
\]
if $|\gs|=k$, then 
\[
D^\gs u_\ep=D^\gs u+o(1)
\]
and, if $k<|\gs|\leq k+j$, then 
\[
D^\gs u_\ep=o(\ep^{k-|\gs|}). 
\]
Notice that if $|\gs|<k$, then $D^\gs u$ is Lipschitz continuous on $\gO$. 
\end{remark}

It is well known that, to obtain an existence result like e.g. 
Proposition \ref{exisnew}, one can apply the Perron method to the boundary value problem  \problem  (see, for instance,  \cite[Remark 4.5]{CIL}).  A crucial property in this construction is  that classical solutions (twice continuously differentiable solutions in the pointwise sense) are also solution in the viscosity sense.

A result in this direction is the following. 

\begin{proposition}\label{equivalenza}
Assume  $g^\pm\in C^1(\ol\gO,\R)$, (H1), (H2) and \eqref{eq5.2111}.
There exists $\ep_*\in(0,\ep_0)$ such that for  $0<\ep<\ep_*$, if $u\in C^2(\ol{\gO_\ep})$ is a classical sub (resp., super) solution of problem \oblique, 

then it is a  viscosity sub (resp., super) solution of \oblique. \end{proposition}

This proposition is a straightforward consequence  of the following Lemma (see also \cite[Remark 4]{BBI}).

\begin{lemma}\label{maximum}
Assume  $g^\pm\in C^1(\ol\gO,\R)$, (H2),  and \eqref{eq5.2111}.
There exist constants $\ep_*\in(0,\ep_0)$ and $\gth>0$ such that if $0<\ep<\ep_*$, $\psi\in C^1(\ol{\gO_\ep})$, and $\psi$ takes a maximum at $c\in \pl_\L \gO_\ep \cap \pl_\T \gO_\ep$ 
(resp., $c\in \pl_\L \gO_\ep \cap \pl_\B \gO_\ep$), then 
\[
(\gamma+\gth\gamma^+)\cdot D\psi(c)\geq 0 
\quad (\text{ resp., } (\gamma+\gth\gamma^-)\cdot D\psi(c)\geq 0\ ).
\]
\end{lemma}

We present a brief proof of Proposition \ref{equivalenza}, based on Lemma \ref{maximum}, for completeness.

\bproof[Proof of Proposition \ref{equivalenza}] We only consider the claim about the 
subsolution property. Let $u\in C^2(\ol{\gO_\ep})$ be a classical subsolution of \oblique. 
Let $\phi\in C^2(\ol{\gO_\ep})$ and $c\in \ol{\gO_\ep}$. Assume that $u-\phi$ takes a
maximum at $c$. If $c\in\gO_\ep$, then the degenerate ellipticity of $F$, a consequence of (H1), 
assures that $F(D^2\phi(c),D\phi(c),u(c),c)\leq F(D^2u(c),Du(c),u(c),c)$, which yields
$F(D^2\phi(c),D\phi(c),u(c),c)\leq 0$. If $c\in\pl_\T\gO_\ep\stm\pl_\L\gO_\ep$, 
then, by \eqref{oblique_TB}, $\gamma^+$ is oblique to $\pl_\T\gO_\ep$, 
which implies that $\gamma^+\cdot D(\phi-u)(c) \leq 0$, and hence $\gamma^+\cdot D\phi(c)\leq \beta^+(c)$. Similarly, $\gamma^-$ and  $\gamma$ are oblique to $\pl_\B\gO_\ep$ at $c\in\pl_\B\gO_\ep\stm \pl_\L\gO_\ep$ and to $\pl_\L\gO_\ep$ at $c\in\pl_\L\gO_\ep\stm(\pl_\T\gO_\ep\cup\pl_\B\gO_\ep)$, respectively, and hence, if $c\in \pl_\B\gO_\ep\stm\pl_\L\gO_\ep$, then $\gamma^-\cdot D\phi(c)\leq 
\beta^-(c)$
and if $c\in \pl_\L\gO_\ep\stm(\pl_\T\gO_\ep\cup\pl_\B\gO_\ep)$, then $\gamma\cdot D\phi(c)\leq \beta(c)$. 
Finally, we consider the case $c\in (\pl_\T\gO_\ep\cup \pl_\B\gO_\ep)\cap\pl_\L\gO_\ep$. 
If $c\in\pl_\T\gO_\ep\cap\pl_\L\gO_\ep$, then, by Lemma \ref{maximum}, we have
$(\gamma+\gth\gamma^+)\cdot D(u-\phi)(c)\geq 0$ for some $\gth>0$. This
implies either $\gamma\cdot D(u-\phi)(c)\geq 0$ or $\gamma^+\cdot D(u-\phi)(c)\geq 0$. 
Here, the former and the latter yield  
$\gamma\cdot D\phi(c)\leq \beta(c)$ and $\gamma^+\cdot D\phi(c)\leq \beta^+(c)$, respectively.
If, instead, $c\in\pl_\L\gO_\ep\cap \pl_\B\gO_\ep$, then we have
$(\gamma+\gth\gamma^-)\cdot D(u-\phi)(c)\geq 0$ for some $\gth>0$ by Lemma \ref{maximum}. 
Hence, similarly to the above, we have either $\gamma\cdot D\phi(c)\leq \beta(c)$ or 
$\gamma^-\cdot D\phi(c)\leq \beta^-(c)$. Thus, we conclude that $u$ is a viscosity subsolution 
of problem \oblique. 
\eproof

\bproof [Proof of Lemma \ref{maximum}]
We use the notation that, given a vector $\gz=(\bar\gz_1,\ldots,\bar\gz_{N+1})\in\R^{N+1}$, we write $\gz_1=(\bar\gz_1,\ldots,\bar\gz_N)\in\R^N$ and $\gz_2=\bar\gz_{N+1}\in\R$, 
so that $\gz=(\gz_1,\gz_2)$. 
 
\vspace{0.3cm}

It is enough to prove the following

{\bf Claim:} {\it There exist constants $\ep_*\in(0,\ep_0)$ and $\gth>0$ such that if 
$0<\ep<\ep_*$, then  
\begin{enumerate}\item[(i)] for any $c\in \pl_\L \gO_\ep    \cap \pl_\T \gO_\ep$, there is $t_c>0$ such that 
$c-t(\gamma+\gth\gamma^+) \in\gO_\ep$ for $0<t<t_c$, \vspace{-8pt}
\item[(ii)] for any $c\in \pl_\L \gO_\ep  \cap\pl_\B \gO_\ep$, there is $t_c>0$ such that 
$c-t(\gamma+\gth\gamma^-) \in\gO_\ep$ for $0<t<t_c$.
\end{enumerate}
}
\vspace{0.3cm}
 Choose $\rho\in C^1(\R^N,\R)$ so that
\[
\rho<0 \ \ \IN\gO,\quad \rho>0 \ \ \IN\R^N\stm\ol\gO, \ \ \AND \ \ D\rho\not=0 \ \ \ON\pl\gO. 
\]
We may regard $\rho$ as a function of $(x,y)\in\R^{N+1}$ independent of $y$.  
Define $\rho^\pm(x,y)=\pm(y-\ep g^\pm(x))$ \ on $\R^{N+1}$. We only need to show that 
for a choice of $\ep_*\in(0,\ep_0)$ and $\gth>0$, \begin{enumerate}
\item[(iii)] $(\gamma+\gth\gamma^+)\cdot D\rho(c)>0$ and 
$(\gamma+\gth\gamma^+)\cdot D\rho^+(c)>0$ \ \FOR $c\in \pl_\L \gO_\ep    \cap \pl_\T \gO_\ep   $ and $0<\ep<\ep_*$,  
\item[(iv)]$(\gamma+\gth\gamma^-)\cdot D\rho(c)>0$ and 
$(\gamma+\gth\gamma^-)\cdot D\rho^-(c)>0$ \  \FOR $c\in \pl_\L \gO_\ep  \cap\pl_\B \gO_\ep $ and $0<\ep<\ep_*$.
\end{enumerate}
Fix $\ep\in(0,\ep_0)$ and $\gth>0$.  
For $c=(c_1,c_2)\in  \pl_\L \gO_\ep    \cap \pl_\T \gO_\ep   $, we compute that
\beq\label{eq.go.3}\bald
(\gamma+\gth\gamma^+)\cdot D\rho(c)&=(\gamma_1+\gth \gamma_1^+,\gamma_2+\gth)\cdot(|D\rho(c_1)|\nu(c_1),0)
\\&= |D\rho(c_1)|(\gamma_1(c)\cdot\nu(c_1)+\gth \gamma^+_1(c)\cdot \nu(c_1)),
\eald
\eeq
and
\beq\label{eq.go.4}\bald
(\gamma+\gth\gamma^+)\cdot D\rho^+(c)
&=(\gamma+\gth\gamma^+)\cdot (-\ep Dg^+(c_1),1)
\\&=-\ep (\gamma_1+\gth\gamma_1^+)\cdot Dg^+(c_1)+(\gamma_2+\gth )(c).
\eald
\eeq
Similarly, for $c\in \pl_\L \gO_\ep  \cap\pl_\B \gO_\ep $, we compute that
\beq\label{eq.go.5}\bald
(\gamma+\gth\gamma^-)\cdot D\rho(c)&=(\gamma_1+\gth \gamma_1^-,\gamma_2-\gth)\cdot(|D\rho(c_1)|\nu(c_1),0)
\\&= |D\rho(c_1)|(\gamma_1(c)\cdot\nu(c_1)+\gth \gamma^-_1(c)\cdot \nu(c_1)),
\eald
\eeq
and
\beq\label{eq.go.6}\bald
(\gamma+\gth\gamma^-)\cdot D\rho^-(c)
&=(\gamma+\gth\gamma^-)\cdot (\ep Dg^-(c_1),-1)
\\&=\ep (\gamma_1+\gth\gamma_1^-)\cdot Dg^-(c_1)-(\gamma_2-\gth )(c).
\eald
\eeq
Since 
\[
\min_{(x,y)\in\pl\gO\tim[-1,1]}\gamma_1(x,y)\cdot \nu(x)>0 \ \ \text{ by \eqref{eq5.2111}},
\]
we may choose $\gth>0$ small enough so that 
\[
\min_{(x,y)\in \pl\gO\tim[-1,1]}\min\{(\gamma_1(x,y)+\gth \gamma^+_1(x,y))\cdot \nu(x), 
(\gamma_1(x,y)+\gth\gamma_1^-(x,y))\cdot\nu(x)\}>0. 
\]
Consequently, from \eqref{eq.go.3} and \eqref{eq.go.5}, we find that 
\[
(\gamma+\gth\gamma^+)\cdot D\rho(c)>0 \ \ \FOR c\in  \pl_\L \gO_\ep  \cap\pl_\T \gO_\ep  ,   
\]
and
\[
(\gamma+\gth\gamma^-)\cdot D\rho(c)>0 \ \ \FOR c\in \pl_\L \gO_\ep  \cap\pl_\B \gO_\ep.  
\]

Next, thanks to \eqref{eq5.2111}, we can choose $\gd\in(0,1)$ so that 
\[
\max_{\pl\gO\tim[-\gd,\gd]}|\gamma_2(c)|\leq 
\fr{\gth  
}{2}. 
\]
We choose $\ep_1\in(0,\ep_0)$ so that 
\[
\ep_1\max\{|g^+(x)|,|g^-(x)|\}\leq \gd \ \ \FOR x\in\ol\gO,
\]
which implies that for $0<\ep\leq\ep_1$, 
\[
\ol{\gO_\ep} \subset \ol\gO\tim[-\gd,\gd],
\]
and moreover,    
\[
|\gamma_2(c)|\leq \fr{\gth}{2} \ \ \ \FOR c\in   (\pl_\L \gO_\ep  \cap\pl_\T \gO_\ep) \cup \:   (\pl_\L \gO_\ep  \cap\pl_\B \gO_\ep) \ \AND\  0<\ep\leq\ep_1. 
\]
Now, from \eqref{eq.go.4} and \eqref{eq.go.6}, we find that for $0<\ep<\ep_1$, if $c\in \pl_\L \gO_\ep  \cap\pl_\T \gO_\ep  $, then
\[
(\gamma+\gth\gamma^+)\cdot D\rho^+(c)\geq \fr{\gth}{2}-\ep (\gamma_1+\gth\gamma_1^+)\cdot Dg^+(c_1)
\]
and if $c\in  \pl_\L \gO_\ep  \cap\pl_\B \gO_\ep$, 
\[
(\gamma+\gth\gamma^-)\cdot D\rho^-(c)\geq \fr{\gth}{2}+\ep (\gamma_1+\gth\gamma_1^-)\cdot Dg^-(c_1).
\]   
We select $\ep_*\in(0,\ep_1)$ so that 
\[
\ep_* \max_{c\in\pl\gO\tim[-\gd,\gd]} [(\gamma_1+\gth\gamma_1^+)(c)\cdot Dg^+(c_1)]^+
<\fr{\gth}{2},
\]
and
\[
\ep_* \max_{c\in\pl\gO\tim[-\gd,\gd]} [(\gamma_1+\gth\gamma_1^-)(c)\cdot Dg^-(c_1)]^-
<\fr{\gth}{2},
\]
to obtain for $0<\ep<\ep_*$,
\[
(\gamma+\gth\gamma^+)\cdot D\rho^+(c)>0 \ \ \FOR c\in  \pl_\L \gO_\ep  \cap\pl_\T \gO_\ep,  
\]
and 
\[
(\gamma+\gth\gamma^-)\cdot D\rho^-(c)>0 \ \ \FOR c\in  \pl_\L \gO_\ep  \cap\pl_\B \gO_\ep. 
\]
This proves the claim and hence the lemma.   
\eproof

In the case of the Dirichlet boundary condition \eqref{eq.lbc.2}, as the next proposition states, 
classical sub and super solutions are also viscosity sub and super solutions, respectively.
The corners of the domain do not matter.
  
\begin{proposition}\label{equivalenzaD} 
Assume  $g^\pm\in C^1(\ol\gO,\R)$, (H1), (H2).
Let $\ep\in(0,\ep_0)$.  If $u\in C^2(\ol{\gO_\ep})$ is a classical sub (resp., super) solution of problem \dirichlet, 
then it is a  viscosity sub (resp., super) solution of \dirichlet. 
\end{proposition}

\bproof We only treat the claim about the subsolution property. Let $u\in C^2(\ol{\gO_\ep})$ be a 
classical subsolution to \dirichlet. Assume that for some $\phi\in C^2(\ol{\gO_\ep})$ 
and $c\in \ol{\gO_\ep}$, $u-\phi$ takes a maximum at $c$. As in the proof of Proposition
\ref{equivalenza}, thanks to the degenerate ellipticity of $F$ and the oblique property 
of $\gamma^\pm$, as stated as \eqref{oblique_TB}, we easily see that if $c\in\gO_\ep $, then 
$F(D^2\phi(c),D\phi(c),u(c),c)\leq 0$, if $c\in\pl_\T\gO_\ep\stm \pl_\L\gO_\ep$, then 
$\gamma^+\cdot D\phi(c)\leq \beta^+(c)$, and if $c\in \pl_\B\gO_\ep\stm 
\pl_\L\gO_\ep$, then $\gamma^-\cdot D\phi(c)\leq \beta^-(c)$. It remains to check the case when 
$c\in \pl_\L\gO_\ep$. However, if $c\in\pl_\L\gO_\ep$, then we have $u(c)\leq \beta(c)$, which 
is enough to conclude that $u$ is a viscosity subsolution to \dirichlet.
\eproof

\bproof[Proof of Proposition \ref{exisnew}]  First, we build classical strict sub and super solutions to 
\problem . Choose functions  $\rho, h\in C^2(\ol\gO,\R)$ such that $h$ satisfies
$$ g^-(x)<h(x)<g^+(x) \ \ \ \FOR x\in\ol\gO. $$
while $\rho$ will be chosen depending on the lateral boundary condition. 
We choose a constant $\gd_0>0$ so that
\beq\label{eq3.0}
\pm(g^\pm(x)-h(x))>\gd_0 \ \ \FOR x\in\ol\gO.  
\eeq
As suggested by the formal expansion argument in the previous section, we set
\[
\overline v(x)=\beta_o(x)-\gamma_o(x)\cdot D\rho(x),  \ \ \underline v(x)=\beta_o(x)+\gamma_o(x)\cdot D\rho(x) \ \ \ \FOR x\in\ol\gO, 
\]
where $\beta_o, \gamma_o$ are defined by \eqref{eq1.4} and we choose  
families $\{\ol v_\ep\mid 0<\ep<\ep_0\}$ and $\{\ul v_\ep\mid 0<\ep<\ep_0\}$ of $C^2$ functions, approximating $\ol v$ and $\ul v$, respectively. 
According to Lemma \ref{lem1} and Remark \ref{rem1}, we can choose $\{\ol v_\ep \}$ and $\{\ul v_\ep \}$ so that 
as $\ep\to 0+$,
\[
\ol v_\ep=\ol v+O(\ep), \quad D\ol v_\ep=D\ol v+o(1), 
\quad\AND\quad \quad D^2\ol v_\ep=o(\ep^{-1}). 
\]
\[
\ul v_\ep=\ul v+O(\ep), \quad D\ul v_\ep=D\ul v+o(1), 
\quad\AND\quad \quad D^2\ul v_\ep=o(\ep^{-1}). 
\]

 For a constant $\gL>0$, to be fixed later, we can
define $\ol\psi_\ep,\, \ul\psi_\ep\in C^2(\ol\gO\tim[-1,1])$ by putting
\[
\ol\psi_\ep(x,y)=\rho(x)+\ol v_\ep(x)y+\fr{\gL}{2}(y-\ep h(x))^2 \ \ \FOR (x,y)\in\ol\gO\tim[-1,1]
\]                       
and
\[
\ul\psi_\ep(x,y)=-\rho(x)+\ul v_\ep(x)y-\fr{\gL}{2}(y-\ep h(x))^2 \ \ \FOR (x,y)\in\ol\gO\tim[-1,1].
\]
The function $\rho$ will be chosen later in order that $\ol\psi_\ep$ and $\ul \psi_\ep$ satisty the lateral boundary condition.
When ever we need to treat both $\ol\psi_\ep$  and $\ul\psi_\ep$ we will use $\psi_\ep$ idem for $\ol v_\ep$ and $\ul v_\ep$.

Note by the definition of $\ol v$ that
\[\bald
(\gamma^+_1\cdot  D\rho +\ol v_\ep-\beta^+)(x,\ep g^+(x))
&=(\gamma^+_1\cdot D\rho +\ol v-\beta^+)(x,\ep g^+(x))+O(\ep)
\\&=(\gamma_o\cdot D\rho + \ol v-\beta_o)(x)+O(\ep)
=O(\ep),
\eald
\]
and
\[\bald
(\gamma^-_1\cdot  D\rho -\ol v_\ep-\beta^-)(x,\ep g^-(x))
&=(\gamma^-_1\cdot D\rho -\ol v-\beta^-)(x,\ep g^-(x))+O(\ep)
\\&=(-\gamma_o\cdot D\rho - \ol v+\beta_o)(x)+O(\ep)
=O(\ep),
\eald
\]
where we have used that by \eqref{eq1.6},  $(\gamma^\pm_1,\beta^\pm)(x,\ep g(x))=(\pm\gamma_o,\pm\beta_o)(x)+O(\ep)$, (for $g= g^+$ and $g= g^-$).  Similarly, we see that 
\[\bald
(-\gamma^+_1\cdot  D\rho +\ul v_\ep-\beta^+)(x,\ep g^+(x))
=(-\gamma_o\cdot D\rho + \ul v-\beta_o)(x)+O(\ep)
=O(\ep),
\eald
\]
and
\[\bald
(-\gamma^-_1\cdot  D\rho -\ul v_\ep-\beta^-)(x,\ep g^-(x))
&=(\gamma_o\cdot D\rho - \ul v+\beta_o)(x)+O(\ep)
=O(\ep).
\eald
\]

We now compute
\[
D\ol \psi_\ep(x,y)=( D\rho(x)+ yD\ol v_\ep(x) + \gL\ep(\ep h(x)-y)Dh(x), \ol v_\ep(x)+\gL(y-\ep h(x)).
\]
Hence, 
\[\bald
&(\gamma^\pm \cdot D\ol \psi_\ep-\beta^\pm)(x,\ep g^\pm(x))
=(\gamma_1^\pm\cdot D\rho\pm \ol v_\ep-\beta^\pm)(x,\ep g^\pm(x)) 
\\&\quad +\gamma_1^\pm(x,\ep g^\pm(x))\cdot(\ep g^\pm(x) D\ol v_\ep(x) +\gL\ep(\ep h(x)-\ep g^\pm(x))Dh(x)) 
\\&\quad \pm \gL(\ep g^\pm(x)-\ep h(x))
\\&= \pm\ep \gL (g^\pm(x)-h(x)) +O(\ep).
\eald
\]
Similarly, 
\[(\gamma^\pm \cdot D\ul\psi_\ep-\beta^\pm)(x,\ep g^\pm(x))
= \mp\ep \gL (g^\pm(x)-h(x)) +O(\ep).
\]
These, combined with \eqref{eq3.0}, assure that, for some constant $C_1>0$  independent of $\ep$, 
\[
(\gamma^\pm\cdot D\ol\psi_\ep-\beta^\pm)(x,\ep g^\pm(x))
\geq 
\ep\gL\gd_0-C_1\ep,
\]
and
\[
(\gamma^\pm\cdot D\ul\psi_\ep-\beta^\pm)(x,\ep g^\pm(x))
\leq 
-\ep\gL\gd_0+C_1\ep.
\]
Selecting $\gL>0$ so that $\gd_0\gL>C_1$, we find that 
\beq\label{eq3.1}\left\{\,\bald
&\gamma^+\cdot D\ol \psi_\ep>\beta^+ \ \ \ON \pl_\T\gO_\ep,  
\qquad\quad 
\gamma^-\cdot D\ol\psi_\ep>\beta^- \ \ \ON \pl_\B\gO_\ep,  
\\& \gamma^+\cdot D\ul \psi_\ep<\beta^+ \ \ \ON \pl_\T\gO_\ep, 
\quad \AND\quad  \gamma^-\cdot D\ul\psi_\ep<\beta^- \ \ \ON \pl_\B\gO_\ep.   
\eald
\right.
\eeq

Next, we need to choose $\rho$ in order to satisfy the lateral condition. Hence we need to treat two different cases:

For problem \oblique we select $\rho\in C^1(\ol\gO)$ so that 
$D\rho(x)/|D\rho(x)|=\nu(x)$ for $x\in\pl\gO$. Since $\gamma_1\cdot \nu>0$ on $\pl\gO\tim[-1,1]$, 
we have $\gamma_1\cdot D \rho >0$ on $\pl\gO\tim[-1,1]$. By approximating $\rho$ by a smooth 
function, we may assume that $\rho\in C^2(\ol\gO)$ and $\gamma_1\cdot D\rho>0$ on $\pl\gO\tim[-1,1]$. Multiplying $\rho$ by a large constant and replacing $\rho$ with the resulting 
function, 
we may assume that 
$$\gamma_1 \cdot D\rho>|\beta|+1\ \  \mbox{on }\ \pl\gO\tim[-1,1].$$
We compute that for any $x\in\pl\gO$, 
\[
\gamma \cdot D\ol \psi_\ep
= \gamma_1\cdot D\rho+y\gamma_1\cdot D\ol v_\ep (x)+ \ep \gL (\ep h(x)-y)\gamma_1\cdot Dh(x) +\gamma_2\ol v_\ep + O(\ep)
\] 
\[
=\gamma_1\cdot D\rho +o(1)>\beta
\] 
Similarly
\[
\gamma \cdot D\ul \psi_\ep=-\gamma_1\cdot D\rho +o(1)<\beta.
\]

For problem \dirichlet instead we choose $\rho\in C^2(\ol\Omega)$ such that
$$\rho>|\beta|+1$$
and the lateral Dirichlet condition is satisfied.

Observe that $\ol \psi_\ep$ and $\ul\psi_\ep$ have the same regularity, as functions on $\ol{\gO_\ep}$. 
We need to prove that
\[
\psi_\ep=O(1), \quad D\psi_\ep=O(1), \quad\AND\quad D^2\psi_\ep=O(1).
\]
In this regard, the only terms in $\psi_\ep$ that need to be taken care of are
$f_\ep(x,y):=y\ol v_\ep(x)$ or $f_\ep(x,y):=y\ul v_\ep(x)$ but in both cases
\[\left\{\,\bald
&f_\ep(x,y)=O(\ep), \quad Df_\ep(x,y)=(yDv_\ep(x),v_\ep(x))=O(1), 
\\& D_yDf_\ep(x,y)=(Dv_\ep, 0)=O(1), \quad 
D_y^2f_\ep(x,y)=0, 
 \\& D_x^2f_\ep(x,y)=yD^2v_\ep(x)=O(\ep)o(\ep^{-1})=o(1). 
\eald\right. 
\]
Consequently, 
\[
F(D^2\psi_\ep, D\psi_\ep,\psi_\ep,x,y)=O(1)
\]
as functions on $\ol{\gO_\ep}$, and there is a constant $C_3>0$ such that 
\[
|F(D^2\psi_\ep, D\psi_\ep,\psi_\ep,x,y)|\leq C_3 \ \ \ \FOR 
(x,y)\in\ol{\gO_\ep} \ \AND\ 0<\ep<\ep_1. 
\]

For any $M>0$, by (H1), we have for $(x,y)\in\ol{\gO_\ep}$,
\[\bald
&F(D^2\ol \psi_\ep, D\ol \psi_\ep,\ol \psi_\ep+M,x,y)\geq -C_3 +\alpha M,
\\&F(D^2\ul \psi_\ep, D\ul \psi_\ep,\ul \psi_\ep-M,x,y)\leq C_3 -\alpha M.
\eald
\]
By selecting $M>0$ so that $\alpha M>C_3$, recalling \eqref{eq3.1} and the appropriate choice of $\rho$ that implies the lateral condition,
we find that for every $\ep\in(0,\ep_1)$, the functions $\ol \psi_\ep +M$ and $\ul \psi_\ep-M$  are classical, strict, sub 
and supersolutions to \problem, respectively.  

Now, let $\ep\in(0,\ep_1)$. 
It follows from 
the classical argument for the maximum principle that if $u^\ep$ is a 
viscosity solution to \problem, then
$\ul \psi_\ep -M\leq u^\ep\leq \ol \psi_\ep+M$ on $\ol{\gO_\ep}$.  
This shows that the family of all viscosity solutions $u^\ep$ to 
\problem, with $\ep\in (0,\ep_1)$, is uniformly bounded 
on $\ol{\gO_\ep}$. That is, there is a constant $C>0$, independent of $\ep>0$, 
such that  if $u^\ep$ is a 
viscosity solution to \problem, then $|u^\ep(x,y)|\leq C$ for all 
$(x,y)\in\ol{\gO_\ep}$ and $\ep\in(0,\ep_1)$. 

Let $\ep_*>0$ be a constant from Proposition \ref{equivalenza} or \ref{equivalenzaD} depending on the lateral condition. 
We may assume that $\ep_1\leq \ep_*$. Let $\ep\in(0,\ep_1)$.  By Proposition \ref{equivalenza} or  \ref{equivalenzaD}, 
the functions $\ol \psi_\ep+M$ and $\ul \psi_\ep-M$ are viscosity sub and super solutions of 
\problem, respectively. 
Thus, for any $\ep\in(0,\ep_1)$, Perron's method readily produces a viscosity solution $u^\ep$ to \problem such that $\ul \psi_\ep-M\leq u^\ep\leq  \ol \psi_\ep+M$ on $\ol{\gO_\ep}$.    Indeed, if we set 
$$
u^\ep(z)= \sup \{  v(z) \mid v  \mbox{ is a viscosity subsolution  of } \eqref{eq1.1}-\eqref{eq1.3},  \ul \psi_\ep-M\leq v \leq  \ol \psi_\ep+M \}
$$
then the function $u^\ep$ is a viscosity solution to \problem. Note that Proposition \ref{equivalenza} or \ref{equivalenzaD} are crucial  when one checks that the lower semicontinuous envelope of $u^\ep$ is a supersolution to \problem.
\eproof

\section{Proof of main result}
We begin by stating Theorem \ref{thm1} in the general case of problem \problem. 
Let $S_\ep$ denote the set of viscosity solutions to \problem, with given 
$\ep>0$.
Under the hypotheses of Proposition \ref{exisnew}, we define the half relaxed limits $u^\pm$ 
by 
\beq \label{defupum}
\bald 
u^+(x)&=\lim_{r\to 0^+}\sup\{u(\xi,\eta)\mid u\in S_\ep,\ (\xi,\eta)\in \ol{\gO_\ep},\ 
0<\ep<r,\ |\xi-x|<r \},
\\ u^-(x)&=\lim_{r\to 0^+}\inf\{u(\xi,\eta)\mid u\in S_\ep,\ (\xi,\eta)\in\ol{\gO_\ep},\ 
0<\ep<r,\ |\xi-x|<r\},
\eald
\eeq
which are bounded functions on $\ol\gO$ by Proposition \ref{exisnew}. 
Let us consider the following problems in $\Omega$
\begin{equation}\label{limeq.lbc.o}
\left\{\begin{array}{lc}
G(D^2u,Du,u,x) =0 & \IN\gO, \\
\gamma \cdot Du=\beta(x,0) \ & \ON\pl\gO
\end{array}
\right.
\end{equation}
or
\begin{equation}\label{limeq.lbc.d}
\left\{\begin{array}{lc}
G(D^2u,Du,u,x) =0 & \IN\gO, \\
u=\beta(x,0) \ & \ON\pl\gO
\end{array}
\right.
\end{equation}
which we also call respectively \obliquezero and \dirichletzero.
\begin{theorem}\label{thm1new} Under the hypotheses of Proposition \ref{exisnew}, 
if  $S_\ep$ are the solutions of problem \oblique, the functions $u^+$ and $u^-$ are a viscosity sub and super solutions to \eqref{limeq.lbc.o}, respectively;

\noindent while
if  $S_\ep$ are the solutions of problem \dirichlet, the functions $u^+$ and $u^-$ are a viscosity sub and super solutions to \eqref{limeq.lbc.d}, respectively.
\end{theorem}
\bproof[Proof of Theorem \ref{thm1new}] We show first that the relaxed limit $u^+$ is a viscosity 
subsolution of \eqref{limeq.lbc.o}. Let $\phi\in C^\infty(\ol\gO)$ and assume that 
$u^+-\phi$ has a strict maximum at a point $\hat x\in\ol\gO$.  

We adapt the calculations in the 
previous section, to modify the test function $\phi$, so that the resulting function $\psi$  satisfies
$\gamma^+_1\cdot D_x\psi+D_y\psi>\beta^+$ on $\pl_\T\gO_\ep$ and 
$\gamma^-_1\cdot D_x\psi-D_y\psi>\beta^-$  on $\pl_\B\gO_\ep$.  
Referring \eqref{eq2.5}, 
\eqref{eq2.5+1}, 
\eqref{eq2.8+5}, and \eqref{a1}, with 
$\phi$ in place of $u^0$, we define 
\[\bald
v(x)&=\beta_o(x)-\gamma_o(x)\cdot D\phi(x), 
\\ w^\pm(x)&=\fr{g^\pm}{(g^+-g^-)} 
(l^\pm-k^\pm\cdot D\phi\mp \gamma_o(x)\cdot Dv), 
\\ \psi_\ep(x,y)&:=\phi(x)+\ep v(x)(y/\ep)+\fr{\ep^2}{2}\big\{(y/\ep -g^-)^2 w^+(x)+(y/\ep -g^+)^2
w^-(x)
\\&\qquad +\gd(y/\ep-h)^2\big\}
\\ &=\phi(x)+v(x)y+\fr{1}{2}\big\{(y -\ep g^-)^2 w^+(x)+(y -\ep g^+)^2
w^-(x)+\gd(y-\ep h)^2\big\},
\eald
\] 
where  $h\in C^2(\ol\gO)$ is a function chosen to satisfy 
\[
g^-(x)<h(x)<g^+(x) \ \ \ \FOR x\in\ol\gO,
\]  
and $\gd$ is an arbitrarily fixed positive constant. Comparison of $\psi_\ep$ above 
to the formal expansion in the previous section indicates that $\psi_\ep$ adds the last term as an additional one,  which will allow us to make the boundary inequalities strict.

By assumption, we have \ $v\in C^1(\ol\gO,\R)$, $w^\pm\in C(\ol\gO,\R)$, 
and $\psi_\ep\in C(\ol\gO\tim[-1,1],\R)$. To enhance the smoothness of $\psi_\ep$, 
we introduce families $\{v_\ep\}$, $\{w_\ep^\pm\}$, and $\{g_\ep^\pm\}$, with $\ep>0$, 
of functions on $\ol\gO$ 
approximating $v$, $w^\pm$, and $g^\pm$, respectively, such that 
\beq\label{eq3.3}\left\{\ \bald
&v_\ep, w_\ep^\pm, g^\pm \in C^\infty(\ol\gO),\quad
\\&v_\ep=v+O(\ep),\quad Dv_\ep=Dv+o(1), \quad D^2v_\ep=o(\ep^{-1}), 
\\& g_\ep^\pm=g^\pm+O(\ep),\quad Dg_\ep^\pm=Dg^\pm+o(1),\quad D^2g^\pm_\ep=o(\ep^{-1}), 
\\& w_\ep^\pm=w^\pm+o(1), \quad Dw_\ep^\pm=o(\ep^{-1}),\quad 
D^2w_\ep^\pm=o(\ep^{-2}).
\eald \right.
\eeq
Such a choice of $v_\ep, w_\ep^\pm$, and $g_\ep^\pm$ is possible thanks to Lemma \ref{lem1} and Remark \ref{rem1}. Moreover, reparametrizing $v_\ep$, we may assume that 
\beq\label{eq3.4}
v_\ep=v+o(\ep),\quad Dv_\ep=Dv+o(1), \quad D^2v_\ep=o(\ep^{-1}).
\eeq

Select $\ep_1\in(0,\ep_0)$ sufficiently small so that for $0<\ep<\ep_1$, 
\[
g_\ep^-+\ep_1<h<g_\ep^+-\ep_1\quad\ON\ \ol\gO.
\]
In what follows we consider only those with parameter $\ep\in(0,\ep_1)$.  
We define $\Psi_\ep\in C^\infty(\ol\gO\tim[-1,1],\R)$ by
\[
\Psi_\ep(x,y)=\phi(x)+v_\ep(x)y+\fr 12\{(y-\ep g^-_\ep)^2w_\ep^++(y-\ep g^+_\ep)^2 w_\ep^- 
+\gd(y-\ep h)^2 \}. 
\]
 
Compute that 
\beq\label{eq3.4+1}\left\{\, \bald
D_x\Psi_\ep&=D\phi +yDv_\ep +\fr 12\{(y-\ep g_\ep^-)^2Dw_\ep^+
+(y-\ep g_\ep^+)^2Dw_\ep^-\}
\\&\quad -\ep\{(y-\ep g_\ep^-)w_\ep^+Dg_\ep^-
+(y-\ep g_\ep^+)w_\ep^-Dg_\ep^++\gd(y-\ep h)Dh\},
\\ D_y\Psi_\ep&=v_\ep +(y-\ep g_\ep^-)w_\ep^+ +(y-\ep g_\ep^+)w_\ep^-+\gd(y-\ep h),
\eald\right. 
\eeq
and observe by \eqref{eq3.3} and \eqref{eq3.4} that for $g=g^+$ and $g=g^-$,
\beq\label{eq3.4+2}\left\{\, \bald
D_x\Psi_\ep(x,\ep g(x))&=D\phi +\ep g Dv+o(\ep),
\\ D_y\Psi_\ep(x,\ep g(x))&=v+\ep(g-g^-)w^++\ep(g-g^+)w^-+\ep\gd(g-h)+o(\ep). 
\eald\right.
\eeq
Hence,   
\[\bald
(\gamma^+_1\cdot D_x\Psi_\ep+D_y\Psi_\ep-\beta^+)(x,\ep g^+(x))
&=\gamma^+_1(x,\ep g^+)\cdot (D\phi+\ep g^+ Dv)
\\&\quad +v+\ep(g^+-g^-)w^+
+\ep\gd (g^+-h)
\\&\quad -\beta^+(x,\ep g^+)+o(\ep) 
\\&= \gamma_o(x)\cdot(D\phi+\ep g^+Dv)+\ep g^+ k^+\cdot D\phi
\\&\quad +v+\ep(g^+-g^-)w^++\ep\gd(g^+-h)
\\&\quad -\beta_o(x)-\ep g^+ l^++o(\ep). 
\eald
\]
By the definition of $v$, we have 
\[
\gamma_o(x)\cdot D\phi+v=\beta_o(x),
\]
and by the definition of $w^+$, we have 
\[
(g^+-g^-)w^+=g^+(l^+-k^+\cdot D\phi -\gamma_o(x)\cdot Dv). 
\]
Combining these together yields
\[
(\gamma^+_1\cdot D_x\Psi_\ep+D_y\Psi_\ep-\beta^+)(x,\ep g^+(x))=\ep \gd (g^+-h)+o(\ep). 
\]
A parallel calculation shows that
\[
(\gamma^-_1\cdot D_x\Psi_\ep-D_y\Psi_\ep-\beta^-)(x,\ep g^-(x))=\ep \gd (h-g^- )+o(\ep). 
\]
Consequently, 
reselecting $\ep_1$ small enough, we may assume that 
for any $0<\ep<\ep_1$, 
\beq\label{eq3.5}
\gamma^+_1\cdot D_x\Psi_\ep+D_y\Psi_\ep>\beta^+ \ \ \ON \pl_\T\gO_\ep
\ \ \AND
\ \ \gamma^-_1\cdot D_x\Psi_\ep-D_y\Psi_\ep>\beta^-\ \ \ON \pl_\B\gO_\ep.
\eeq

By the definition of $u^+$, there exists sequences $\{\ep_j\in(0,\ep_*)\mid j\in\N\}$, 
$\{u_j\in \USC(\ol{\gO_{\ep_j}},\R)\mid j\in\N\}$, and $\{(\xi_j,\eta_j)\in\ol{\gO_{\ep_j}}\}$ 
such that $u_j$ is a viscosity subsolution to \oblique, with $\ep=\ep_j$,
\ $\lim_{j\to\infty}u_j(\xi_j,\eta_j)=u^+(\hat x)$, and $\lim_{j\to \infty}(\xi_j,\eta_j)=(\hat x,0)$. 
For every $j\in\N$, choose a maximum point $(x_j,y_j)\in\ol{\gO_{\ep_j}}$ of $u_j-\Psi_{\ep_j}$. 
We claim that $\lim_{j\to \infty}(x_j,y_j)=(\hat x,0)$. To see this, 
by passing to a subsequence and concentrating on the subsequence, 
we may assume that for some $\bar x\in\ol\gO$, 
$\lim_{j\to \infty}(x_j,y_j)=(\bar x,0)$.   
Noting that 
\[
\lim_{\ep\to 0^+}\max_{(x,y)\in\ol{\gO_\ep}}|\Psi_\ep(x,y)-\phi(x)|=0,
\]
we observe that
\[\bald
&\limsup_{j\to \infty}(u_j-\Psi_{\ep_j})(x_j,y_j)\leq (u^+-\phi)(\bar x),
\\&\liminf_{j\to \infty}(u_j-\Psi_{\ep_j})(x_j,y_j)\geq \liminf_{j\to \infty}(u_j-\Psi_{\ep_j})(\xi_j,\eta_j)
=(u^+-\phi)(\hat x),
\eald
\]
which assures that $\bar x=\hat x$ and our claim is valid. The argument above can be 
used to show that $\lim_{j\to \infty}u_j(x_j,y_j)=u^+(\hat x)$. 

Since $u_j$ is a subsolution to \oblique, with $\ep=\ep_j$, we find in view of 
\eqref{eq3.5} that if $(x_j,y_j)\in\ol{\gO_{\ep_j}}\stm \pl_\L\gO_{\ep_j}$, then 
\beq\label{eq3.6}
F(D^2\Psi_{\ep_j}(x_j,y_j),D\Psi_{\ep_j}(x_j,y_j),u_j(x_j,y_j),x_j,y_j)\leq 0,
\eeq
and if $(x_j,y_j)\in\pl_\L\gO_{\ep_j}$, then either \eqref{eq3.6} or the following holds
\beq\label{eq3.7}
\gamma\cdot D_x\Psi_{\ep_j}(x_j,y_j)\leq \beta(x_j,y_j). 
\eeq

From \eqref{eq3.4+1}, together with \eqref{eq3.3}, 
we obtain on $\ol{\gO_\ep}$,
\[\bald
D_x^2\Psi_{\ep}&=D^2\phi+o(1), 
\\
D_x D_y\Psi_{\ep}&=Dv 
+o(1)
\\
D_y^2\Psi_{\ep}&=w_\ep^+ +w_\ep^-+\gd
\\&=w^++w^-+\gd+o(1). 
\eald
\]
Similarly to \eqref{a2}, we have 
\[
Dv=D\beta_o-D\phi D\gamma_o- \gamma_o D^2\phi. 
\]

From these together, we obtain 
\[\bald
D^2\Psi_\ep&=\begin{pmatrix}D^2\phi & Dv^\T 
\\ 
Dv
& w^++w^-+\gd \end{pmatrix}+o(1)
\\&=\begin{pmatrix}D^2\phi & (D\beta_o-D\phi D\gamma_o-\gamma_o D^2\phi)^\T 
\\ 
D\beta_o-D\phi D\gamma_o -\gamma_o D^2\phi 
& w^++w^-+\gd \end{pmatrix}+o(1),
\eald
\]
and, similarly to \eqref{eq3.4+2},
\[
D\Psi_\ep=(D\phi,v)+O(\ep)=(D\phi,\beta_o-\gamma_o\cdot D\phi)+O(\ep).  
\]

Recalling the definition of $b$ and $c$ and following the computation from \eqref{a4} to \eqref{a5}, 
we find that  
\[
(w^++w^-)(x)=c(x)+b(x)\cdot D\phi(x)+\gamma_o(x)D^2\phi(x)\gamma_o(x)^\T.
\]
Thus, we have 
\[\bald
D^2\Psi_\ep&=\begin{pmatrix}D^2\phi & (D\beta_o-D\phi D\gamma_o-\gamma_o D^2\phi)^\T
\\ 
D\beta_o-D\phi D\gamma_o-\gamma_oD^2\phi 
& c+b\cdot D\phi+\gamma_oD^2\phi \gamma_o^\T+\gd \end{pmatrix}+o(1),
\eald
\]
and we infer from \eqref{eq3.6} and \eqref{eq3.7} that if $\hat x\in\gO$, then
\beq\label{eq3.8}
F(A+B+C,\hat p,u^+(\hat x),\hat x,0)\leq 0,
\eeq
where $\hat p=(D\phi(\hat x),\,\beta_o(\hat x)-\gamma_o(\hat x)\cdot D\phi(\hat x))$, 
\[\bald
A&=\begin{pmatrix}D^2\phi(\hat x)& -D^2\phi(\hat x)\gamma_o(\hat x)^\T \\ 
-\gamma_o(\hat x) D^2\phi(\hat x)& \gamma_o(\hat x) D^2\phi(\hat x) \gamma_o(\hat x)^\T\end{pmatrix}, 
\\ B&=\begin{pmatrix} 0& -(D\phi(\hat x)\gamma_o(\hat x))^\T \\
-D\phi(\hat x)\gamma_o(\hat x) &b(\hat x)\cdot D\phi(\hat x)\end{pmatrix},
\\C&=\begin{pmatrix} 0& D\beta_o(\hat x)^\T
\\ D\beta_o(\hat x)
& c(\hat x)+\gd
\end{pmatrix},
\eald
\]
and if $\hat x\in\pl\gO$, then either \eqref{eq3.8} or the following holds 
\[
\gamma(\hat x,0)\cdot D\phi(\hat x)\leq \beta(\hat x, 0).
\]
Since $\gd>0$ is arbitrary, we conclude that $u^+$ is a viscosity subsolution to \eqref{limeq.lbc.o}. 

An argument parallel to the above shows that $u^-$ is a viscosity supersolution to 
\eqref{limeq.lbc.o}, the detail of which is left to the interested reader. 

The case of problem \dirichlet is also similar.
\eproof

Once the comparison principle is established for 
\problemzero, the above theorem assures the 
uniform convergence of solutions $u^\ep$ to 
\problem on $\ol\gO$ as $\ep\to 0^+$. 

\begin{corollary}\label{cor1} Assume the same hypotheses as  Proposition \ref{exisnew}  and, moreover, 
the comparison principle for \problemzero stated as follows:
\begin{enumerate}
\item[(H4)]If $v$ and $w$ are viscosity sub and super solutions to \problemzero,  
respectively, then
$v \leq w$ on $\ol\gO$.
\end{enumerate} Let $u^\ep$ be a viscosity solutions to \problem 
such that $\sup_{\ol\gO_\ep}|u^\ep|\leq C_0$ for every $0<\ep<\ep_1$ and some 
$\ep_1\in(0,\ep_0)$ and $C_0>0$.  
Then, for a (unique)  viscosity solution $u^0$ to \problemzero, 
we have
\beq\label{eq4.7}
\lim_{\ep\to 0^+} \sup_{(x,y)\in\ol{\gO_\ep}}|u^\ep(x,y)-u^0(x)|=0. 
\eeq
\end{corollary}
 
\bproof Let $u^+$ and $u^-$ be the functions defined by \eqref{defupum}. By Theorem 
\ref{thm1new}, $u^+$ and $u^-$ are viscosity sub snd super solutions to \problemzero, 
respectively.
The comparison principle (H4) implies that $u^+=u^-$, $u^0:=u^+=u^-\in C(\ol\gO,\R)$, 
and $u^0$ is a viscosity solution of \problemzero.  
By the definition of $u^\pm$, we easily deduce that \eqref{eq4.7} holds. 
\eproof

\section{Comparison principles and examples} 

In this section, we are concerned with the operators $F$, for which   
the uniform convergence \eqref{eq4.7} is valid i.e. such that the limit equation enjoys the condition for the validity of the  comparison principle. 
Standard and classical requirements for $G$  are: 
\beq\label{eq5.1} \left\{\,\begin{minipage}{0.85\textwidth}for any $ (X,p,x)\in \cS(N)\tim \R^N\tim\ol\gO,\, r,s\in\R$ and some constant $\alpha>0$,
\[
G(X,p,r,x)-G(X,p,s,x)\geq \alpha(r-s) \quad\IF r\geq s
\] 
\end{minipage}\right.
\eeq
and
\beq\label{eq5.2} 
\left\{\,
\begin{minipage}{0.85\textwidth} for each $R>0$, there is a function $\go_R\mid[0,\infty)\to[0,\infty)$, with $\go_R(0)=0$, 
such that 
\[
G(Y,p,r,y)-G(X,p,r,x)\leq\go_R(\gth |x-y|^2+|x-y|(|p|+1)),
\]
whenever $\gth>0$, $p\in\R^N$, $r\in[-R,R]$, $x,y\in\R^N$, and $X,Y\in\cS(N)$ satisfy
\[
-\gth I_{2N}\leq \begin{pmatrix}X& 0\\ 0 &-Y\end{pmatrix}\leq 
\gth \begin{pmatrix}I_N&-I_N\\-I_N&I_N\end{pmatrix}. 
\]
\end{minipage}
\right.
\eeq

\subsection{Comparison principle} 
Regarding the validity of the comparison principle (H4), 
we follow our previous work \cite{BBI} and need  more conditions on the boundary of $\gO$.
We always assume:
\beq\label{eq5.3} \left\{
\begin{minipage}{0.85\textwidth}
  There are a neighborhood V of $\pl\gO$, relative to $\ol\gO$, and a 
function $\go:[0,\infty)\to[0,\infty)$, with $\go(0^+)=0$, such that
\[
G(X, p, r, x) - G(Y, q, r, x) \leq \go(|X - Y| + |p - q|)
\]
for $X, Y \in \cS(N)$, $p, q \in \R^N$, $r \in \R$, and $x \in V$.
\end{minipage}\right.
\eeq

Moreover, if we are considering on the lateral boundary the general oblique conditions, in addition to \eqref{eq5.2111}, we assume either 
\beq\label{eq5.22} 
\begin{minipage}{0.85\textwidth}
the boundary $\pl\gO$ is of class $C^{2,1}$ and  $\gamma_1(\cdot,0)\in C^{0,1}(\pl\gO)$, 
\end{minipage}
\eeq 
or 
\beq\label{eq5.23} \begin{minipage}{0.85\textwidth}
$\pl\gO$ is of class $C^1$ and $\gamma_1(\cdot,0), \beta(\cdot,0)\in C^{1,1}(\pl\gO)$,
\end{minipage}
\eeq
while, in the  case of standard Neumann conditions on the lateral boundary,  the following suffices: 
\beq\label{eq5.4} \left\{
\begin{minipage}{0.85\textwidth} In addition to the condition that $\pl\gO$ is $C^1$,  there is a constant $r_0>0$ such that
\[
B_{r_0}(x+r_0\nu(x))\cap\gO=\emptyset \ \ \ \FOR x\in\pl\gO. 
\]
\end{minipage}\right.
\eeq

The precise result, according to \cite[Theorem I.2]{Ba93}, \cite[Theorem 2.1]{Is91}, and \cite{CIL}, is the following:

 \begin{lemma} \label{lemCompEx}
Assume \eqref{eq5.1}--\eqref{eq5.3}. 
\begin{itemize}
\item[{\bf (O)}] Assume \eqref{eq5.2111}, and either 
\eqref{eq5.22} or \eqref{eq5.23}.  Then, comparison principle (H4) holds for problem (O$_0$), i.e. \eqref{limeq.lbc.o}. 
\item[{\bf (N)}] Under the assumption \eqref{eq5.4}, the comparison principle (H4) holds
for   (N$_0$), i.e. \eqref{eq1.9}--\eqref{eq1.10}.
\end{itemize}

\end{lemma}

\subsection{Examples} 

\def\cM{\mathcal{M}}
We seek for some conditions on $F$ so that
that $G$, defined by \eqref{eq1.8+1}, satisfies \eqref{eq5.1}--\eqref{eq5.4}. 
To do this, we restrict ourselves to the Bellman-Isaacs  
type operators and consider 
function $F$ 
given by
\beq\label{eq5.5}
F(X,p,r,z)=\inf_{\gl \in L}\sup_{\mu\in M} F_{\gl\mu}(X,p,r,z)
\ \ \ON \cS(N+1)\tim\R^{N+1}\tim\R\tim \ol\gO\tim[-1,1],  
\eeq
where $L,\,M$ are given nonempty sets and 
\[
F_{\gl\mu}(X,p,r,z)=-\tr \gs_{\gl\mu}^T\gs_{\gl\mu} X-b_{\gl\mu}\cdot p+c_{\gl\mu} r-f_{\gl\mu},
\]
with 
\[ \bald
&\gs_{\gl\mu}\in C^{0,1}(\ol\gO\tim[-1,1],\cM(k,N+1)),\quad 
b_{\gl\mu}\in C^{0,1}(\ol\gO\tim[-1,1],\R^{N+1}), 
\quad
\\& c_{\gl\mu},\,f_{\gl\mu}\in C(\ol\gO\tim[-1,1],\R).
\eald
\]
In the above, $\cM(k,N+1)$ denotes the set of real $k\tim(N+1)$ matrices and $k\in\N$ is a
given fixed number.  
We assume the following uniform boundedness and equi-continuity of $\gs_{\gl\mu}, 
b_{\gl\mu}, c_{\gl\mu}, f_{\gl\mu}$: for any $(\gl,\mu)\in L\tim M$ and $z,z'\in\ol\gO\tim[-1,1]$,
\beq\label{eq5.6}\left\{\,\bald
&\max\{|\gs_{\gl\mu}(z)|,|b_{\gl\mu}(z)|,|c_{\gl\mu}(z)|,|f_{\gl\mu}(z)|\}\leq C_F,
\\&\max\{|\gs_{\gl\mu}(z)-\gs_{\gl\mu}(z')|,|b_{\gl\mu}(z)-b_{\gl\mu}(z')|\}\leq C_F|z-z'|,
\\&\max\{|c_{\gl\mu}(z)-c_{\gl\mu}(z')|,|f_{\gl\mu}(z)-f_{\gl\mu}(z')|\}\leq \go_F(|z-z'|), 
\eald\right.
\eeq
where $C_F$ is a positive constant and $\go_F\mid [0,\infty)\to[0,\infty)$ is a function satisfying $\go_F(0^+)=~0$.
Also, assume that for some constant $\alpha>0$,
\beq\label{eq5.7}
c_{\gl\mu}(z)\geq \alpha \ \ \ \FOR (\gl,\mu)\in L\tim M
\ \AND\ z\in\ol\gO\tim[-1,1]. 
\eeq 
With all the structural conditions on $F$ stated above, it is easily seen that $F$ satisfies (H1) and that 
\beq\label{eq5.8}
|F(X,p,r,z)-F(Y,q,r,z)|\leq C_F^2|X-Y|+C_F|p-q|
\eeq
for $X,Y\in\cS(N+1)$, $p,q\in\R^{N+1}$, and $(r,z)\in\R\tim\ol\gO\tim[-1,1]$. 
Moreover, if we write $A(X,x), B(p,x), C(x)$ for $A, B, C\in\cS(N+1)$ in \eqref{eq1.8+2}, respectively, it follows that for some constant $C_1>0$,
\[\bald
&|(A(X,x)+B(p,x)+C(x))-(A(Y,x)+B(q,x)+C(x))|
\\&+|(p,
\beta_0-\gamma_0 \cdot p)-(q,
\beta_0-\gamma_0 \cdot q)|
\leq C_1(|X-Y|+|p-q|)
\eald
\]
for $X,Y\in\cS(N)$ and $p,q\in\R^N$. The constant $C_1$ can be chosen to depend only on 
$\|(
\gamma_0,\beta_0,D\gamma_0,D\beta_0)\|_{L^\infty(\gO)}$. 
It is now clear that \eqref{eq5.8} implies \eqref{eq5.3}. 

Since 
\[
A(X,x)=\begin{pmatrix}I_N \\ -
\gamma_0(x)\end{pmatrix} X \begin{pmatrix}I_N & -
\gamma_0(x)^\T\end{pmatrix},
\]
setting 
\[
\tilde \gs_{\gl\mu}(x):=\gs_{\gl\mu}(x,0)\begin{pmatrix}I_N \\ -
\gamma_0(x)\end{pmatrix} \in \cM(k,N),
\]
we have
\[
\tr [(\gs_{\gl\mu}^\T\gs_{\gl\mu})(x,0) A(X,x)]=\tr [(\tilde\gs_{\gl\mu}^\T \tilde \gs_{\gl\mu})(x)X]. 
\]
Under assumption \eqref{eq1.5}, we see that $\tilde \gs_{\gl\mu}\in C^{0,1}(\ol\gO,\cM(k,N))$ 
and the family $\{\tilde\gs_{\gl\mu}\mid (\gl,\mu)\in L\tim M\}$ is equi-Lipschitz continuous on $\ol\gO$. 
Let $e_1,\ldots,e_N$ denote the standard basis of $\R^{N}$. Define $\tilde b_{\gl\mu}
=(\tilde b_{\gl\mu,1},\ldots,\tilde b_{\gl\mu,N})\in C(\ol\gO,\R^N)$ by
\[
\tilde b_{\gl\mu,i}(x)=\tr [(\gs_{\gl\mu}^\T\gs_{\gl\mu})(x,0)B(e_i,x)]+b_{\gl\mu,i}(x,0)
-b_{\gl\mu,N+1}(x,0) \gamma_0(x)\cdot e_i  
\] 
where $b_{\gl\mu,i}$ denotes the $i^{\text{th}}$ entry of the $(N+1)$--vector $b_{\gl\mu}$. 
In addition to the regularity assumptions \eqref{eq1.3+0} and \eqref{eq1.5}, if we assume 
that 
\beq\label{eq5.9}
\gamma_0\in C^{1,1}(\ol\gO,\R^N)\quad\AND\quad k^\pm\in C^{0,1}(\ol\gO,\R^{N}), 
\eeq
then $\tilde b_{\gl\mu}\in C^{0,1}(\ol\gO,\R^N)$ 
and the family $\{\tilde b_{\gl\mu}\}$ is 
equi-Lipschitz continuous on $\ol\gO$.  
Set 
\[
\tilde f_{\gl\mu}(x)=f_{\gl\mu}(x,0)+\tr [(\gs_{\gl\mu}^\T\gs_{\gl\mu})(x,0)C(x)] 
+b_{\gl\mu,N+1}(x,0)\beta_0(x)\ \ \FOR x\in\ol\gO,
\]  
and note that the family $\{\tilde f_{\gl\mu}\}$ is equicontinous on $\ol\gO$. Note also that 
the family $\{(\tilde \gs_{\gl\mu},\tilde b_{\gl\mu},f_{\gl\mu})\mid (\gl,\mu)\in L\tim M\}$ is 
uniformly bounded on $\ol\gO$. Observe that
\beq\label{eq5.10}
G(X,p,r,x)=\inf_{\gl\in L}\sup_{\mu\in M} (-\tr (\tilde \gs_{\gl\mu}^\T\tilde\gs_{\gl\mu} X)-\tilde b_{\gl\mu}\cdot p 
+c_{\gl\mu}r-\tilde f_{\gl\mu}),
\eeq
and, as is well-known (see e.g., \cite{CIL}), this function satisfies  \eqref{eq5.2}, provided 
that the family $\{(\tilde\gs_{\gl\mu},\tilde b_{\gl\mu},c_{\gl\mu},\tilde f_{\gl\mu})\}$ is uniformly 
bounded,  $\{(\tilde\gs_{\gl\mu},\tilde b_{\gl\mu})\}$ is equi-Lipschitz continuous, 
$\{(c_{\gl\mu},\tilde f_{\gl\mu})\}$ is equi-continuous on $\ol\gO$, which is the case under 
assumptions \eqref{eq5.6} and \eqref{eq5.9}. 

To simplify the presentation, we introduce  the following condition which combines several conditions:
\begin{enumerate}\item[(H5)] The function $F$ 
takes the form \eqref{eq5.5} and the conditions \eqref{eq5.6} and \eqref{eq5.7} are satisfied. Also, the conditions (H2), \eqref{eq1.2+2}, \eqref{eq1.3+0}, \eqref{eq1.4}, \eqref{eq1.5}, \eqref{eq1.6}, and \eqref{eq5.9} are in effect.
\end{enumerate}

\begin{theorem}\label{thmObliNuu} 
Assume (H5). 
\begin{itemize}
\item[{\bf (O)}]  Assume \eqref{eq5.2111},  and either \eqref{eq5.22} or \eqref{eq5.23}.  
For $\ep>0$ sufficiently small, let $u^\ep$ be a viscosity solution to \oblique i.e. \eqref{eq1.1}, \eqref{eq1.2}. 
Let $u^0$ be a unique viscosity solution to \obliquezero i.e.  \eqref{limeq.lbc.o}.
Then, $u^\ep$ converges to $u^0$ uniformly on $\ol\gO$: 
\[\lim_{\ep\to 0^+}\sup_{(x,y)\in\ol{\gO_\ep}}|u^\ep(x,y)-u^0(x)|=0.\] 

\item[{\bf (N)}] 
Assume  \eqref{eq5.4}. For $\ep>0$ sufficiently small, let $u^\ep$ be a viscosity solution to \neumann i.e.
\eqref{eq1.1}--\eqref{eq1.3}.  
Let $u^0$ be a unique viscosity solution to  \neumannzero i.e. \eqref{eq1.9}--\eqref{eq1.10}. Then,
$u^\ep$ converges to $u^0$ uniformly on $\ol\gO$:
\[\lim_{\ep\to 0^+}\sup_{(x,y)\in\ol{\gO_\ep}}|u^\ep(x,y)-u^0(x)|=0.\] 
\end{itemize}
\end{theorem}
 
 \bproof 
Thanks to Lemma \ref{lemCompEx}, this theorem is just an application of  Corollary \ref{cor1}, since,
as noted above, $F$ satisfies (H1), which, together with (H5), completes the hypotheses of Proposition \ref{exisnew}. 
\eproof

\subsubsection{\bf Dirichlet boundary condition} We consider 
the case where \problem 
(resp., \problemzero) corresponds to \dirichlet (resp., \dirichletzero), 
in other words, problem \eqref{eq1.1}, \eqref{eq1.2}, and \eqref{eq.lbc.3} 
(resp., \eqref{limeq.lbc.d}). 
If the Dirichlet condition in \eqref{limeq.lbc.d}
is consistent  with the classical one, 
the comparison principle (H4) for \problemzero
is well understood. 

As in the previous subsections, we restrict ourselves to the Bellman-Isaacs 
type operators and thus keep the structure condition  
that \eqref{eq5.5}, \eqref{eq5.6}, \eqref{eq5.7}, and \eqref{eq5.9} are valid, 
and seek for some cases when  the Dirichlet condition in \eqref{limeq.lbc.d} 
is consistent  with the classical one. 

\begin{lemma} \label{lem5} Assume \eqref{eq5.4}, \eqref{eq5.5}, \eqref{eq5.6}, \eqref{eq5.7}, and \eqref{eq5.9}. 
Furthermore, assume 
that for any $x\in \pl\gO$, 
\beq\label{eq5.30}
\inf_{(\gl,\mu)\in L\tim M}|\nu(x)\bmat I_N & -\gamma_0^\T(x)\emat \gs_{\gl,\mu}^\T(x,0)|>0.  
\eeq
Then, the Dirichlet condition in \eqref{limeq.lbc.d} can be understood in the classical sense. 
\end{lemma}

Notice that $\nu(x)\bmat I_N & -\gamma_0^\T(x)\emat \gs_{\gl,\mu}^\T(x,0)$ 
is a $1\tim k$ matrix, i.e., a $k$-vector. Recall that we are always assuming that $\beta\in C(\pl\gO \tim[-1,1],\R)$. 

\begin{remark} Condition \eqref{eq5.30} reads 
$$\inf_{(\gl,\mu)\in L\tim M}\tr[\tilde \gs_{\gl\mu}^\T\tilde\gs_{\gl\mu}(x)\nu(x)\otimes\nu(x)]>0
\quad \FORALL x\in\pl\gO.$$ (See also \eqref{eq5.32} below.) This condition can be thought of 
as a strict ellipticity of the operator $G$ in the normal direction on the boundary $\pl\gO$. 
\end{remark}

\bproof The conclusion of the lemma states precisely that if $u$ is an upper semicontinuous subsolution (resp., lower semicontinuous supersolution) to (D$_0$),  then $u(x)\leq \gb(x,0)$ (resp., $u(x)\geq \gb(x,0)$) for any $x\in\pl\gO$.  We give only the proof in the case when 
$u$ is such a subsolution to (D$_0$), 
the other case is treated similarly. 

We fix any $z\in\pl\gO$.  Thanks to \eqref{eq5.4} (the exterior sphere condition), by reselecting 
$r_0>0$ small enough if necessary, we have 
\beq\label{eq5.31}
\ol{B_{r_0}(z+r_0\nu(z))}\cap\ol{\gO}=\{z\}. 
\eeq
Due to \eqref{eq5.30}, we infer by continuity that for some $\gd>0$ and all 
$(\xi, x)\in B_\gd(\nu(z))\tim (B_\gd(z)\stm\ol\gO)$, 
\[ 
\inf_{(\gl,\mu)\in L\tim M}|\xi \bmat I_N&-\gamma_0^\T(x)\emat\gs_{\gl\mu}(x,0) |>\gd,
\] 
which reads 
\beq
\label{eq5.32}
\inf_{(\gl,\mu)\in L\tim M}\tr [\tilde\gs_{\gl \mu}^\T\tilde\gs_{\gl \mu}(x)\xi\otimes \xi]>\gd^2 . 
\eeq
For the brevity of notation, we set $y:=z+r_0\nu(z)$.
Choose $\ep\in(0,\gd)$ so that 
\[
\fr{y-x}{r_0}\in B_\gd(\nu(z)) \ \ \FOR x\in B_\ep(z).
\]
By \eqref{eq5.31}, we can choose $r_1>r_0$, close enough to $r_0$, 
so that 
\[
\ol{B_{r_1}(y)}\cap\ol{\gO}\subset B_\ep(z). 
\]
Under this choice of $r_1$, we have by \eqref{eq5.32}, 
\beq\label{eq5.33} 
\inf_{(\gl,\mu)\in L\tim M}\tr[\tilde\gs_{\gl \mu}\tilde\gs_{\gl \mu}(x)(x-y)\otimes (x-y)]>r_0^2\gd^2 
\ \ \FORALL x\in B_{r_1}(y)\cap\ol \gO. 
\eeq

Now, we fix any $b>\beta(z,0)$ and will show that $u(z)\leq b$, which is enough  
to complete the proof. 
We may assume, by reselecting $\ep>0$ small enough if necessary,  that 
$b>\beta(x,0)$ for all $x\in B_\ep(z)\cap \pl\gO$.   For $t>0$, we define the smooth function 
$\psi_t$ in $\R^N$ by
\[
\psi_t(x)=b+C_t(1-e^{t(r_0^2-|x-y|^2)}), \quad\text{ with } C_t:=e^{t(r_1^2-r_0^2)}. 
\]
Note that, by \eqref{eq5.31},
\[
\min_{\ol\gO}\psi_t=\psi_t(z)=b,
\]
and 
\[
\min_{\ol\gO\stm B_{r_1}(y)}\psi_t =b+C_t(1-e^{t(r_0^2-r_1^2)}) \to +\infty \ \ \text{ as } t\to +\infty.
\]
Since $u$ is bounded above in $\ol\gO$, this last observation allows us to choose $T>0$ so that for $t\geq T$,
\beq\label{eq5.34}
\max_{\ol\gO}u<\min_{\ol\gO\stm B_{r_1}(y)}\psi_t. 
\eeq
Henceforth, we consider only those $\psi_t$ with $t\geq T$. 

We consider the case where $u\leq\psi_t$ in $\ol\gO$ for some $t\geq T$.  Then we have $u(z)\leq\psi_t(z)=b$, and we finish the proof. 
This is only the case we have. Indeed, otherwise, we get a contradiction.  
To check this, suppose 
\[
\max_{\ol\gO}(u-\psi_t)>0 \ \ \FORALL t\geq T.
\]  
This and \eqref{eq5.34} imply 
\[
\max_{\ol\gO}(u-\psi_t)=\max_{\ol{B_{r_1}(y)}\cap \ol\gO}(u-\psi_t)>0. 
\]
Let $x_t\in \ol{B_{r_1}(y)}\cap \ol\gO$ be a maximum point of $u-\psi_t$. 
Since $x_t\in B_\ep(z)$, if $x_t\in\pl\gO$, then   
\[
u(x_t)>\psi_t(x_t)\geq b>\beta(x_t,0).
\]
Hence, the subsolution property of $u$ yields
\beq\label{eq5.35}
0\geq G(D^2\psi_t(x_t),D\psi_t(x_t), u(x_t),x_t)
\geq G(D^2\psi_t(x_t),D\psi_t(x_t),b,x_t).
\eeq

\noindent Computing, for any $x\in\R^N$, 
\[\bald
D\psi_t(x)&=2tC_t e^{t(r_0^2 -|x-y|^2)}(x-y),
\\ D^2\psi_t(x)&=C_t e^{t(r_0^2-|x-y|^2)} (2t I_N-4t^2(x-y)\otimes (x-y)),
\eald\]
and moreover,  using \eqref{eq5.33} and that $|x_t-y|\leq r_1$:
\[\bald
-\tr[\tilde\gs_{\gl \mu}^\T&\tilde\gs_{\gl \mu}(x_t) D^2\psi_t(x_t)] -\tilde b_{\gl \mu}(x_t)\cdot D\psi_t(x_t)
\\&= tC_t e^{t(r_0^2-|x_t-y|^2)} \big(-2\tr[\tilde\gs_{\gl \mu}^\T\tilde \gs_{\gl\mu}(x_t) ]
+4t \tr[\tilde\gs_{\gl\mu}^\T\tilde\gs_{\gl\mu}(x_t)(x_t-y)\otimes(x_t-y)] 
\\&\quad -2\tilde b_{\gl\mu}(x_t)\cdot (x_t-y)
\big)
\\&\geq  tC_t e^{t(r_0^2-|x_t-y|^2)} \big(-2 B 
+4 r_0^2\gd^2 t
-2 r_1 B \big), 
\eald
\]
where $B>0$ is a constant chosen so that
\[
\sup_{(x,\gl,\mu)\in\ol\gO\tim L\tim M}\max\{|\tilde b_{\gl\mu}(x)|,\,\tr[\tilde \gs_{\gl\mu}^\T \tilde\gs_{\gl\mu}(x)]\}\leq B. 
\]
Observe that for $t$ sufficiently large, $\ 4r_0^2\gd^2 t-2(1+r_1)B\geq 1\ $ and 
\[\bald
-\tr[\tilde\gs_{\gl \mu}^\T&\tilde\gs_{\gl \mu}(x_t) D^2\psi_t(x_t)] -\tilde b_{\gl \mu}(x_t)\cdot D\psi_t(x_t)
\\&\geq  tC_t e^{t(r_0^2-|x_t-y|^2)}\geq t C_t e^{t(r_0^2-r_1^2)}=t. 
\eald
\]
This implies that, as $t\to +\infty$,
\[
G(D^2\psi_t(x_t),D\psi_t(x_t),b,x_t)\to +\infty,
\]
which contradicts \eqref{eq5.35}. 
\eproof

\begin{theorem} Assume (H5), {\eqref{eq5.4},} and \eqref{eq5.30}. 
For $\ep>0$ sufficiently small, let $u^\ep$ be a viscosity solution to \dirichlet.  
Then  $u^\ep$ converges to $u^0$ uniformly on $\ol\gO$: 
\[
\lim_{\ep\to 0^+} \sup_{(x,y)\in\ol{\gO_\ep}}|u^\ep(x,y)-u^0(x)|=0. 
\]

\end{theorem}

\bproof

Since $G$ satisfies \eqref{eq5.1} and \eqref{eq5.2} and the Dirichlet boundary condition 
in \eqref{limeq.lbc.d} can be considered in the classical sense thanks to Lemma \ref{lem5}, it is well-known (see, for instance, \cite{CIL})
that the comparison principle (H4) holds for {(D$_0$).} 
So we just need to apply Corollary~\ref{cor1}. 
\eproof


\begin{thebibliography}{99}
 \bibitem{ABuPe} E. Acerbi, G. Buttazzo and D. Percivale, {\it Thin inclusions in linear elasticity: a variational approach. } J. Reine Angew. Math. 386 (1988), 99-115.
\bibitem{AnBaPe} G. Anzellotti, S. Baldo, D. Percivale, {\it Dimension reduction in variational problems, asymptotic development in $\Gamma$-convergence and thin structures in elasticity. } Asymptotic Anal. 9 (1994), no. 1, 61-100.
\bibitem{AP} Jos\'e M. Arrieta, Marcone C.  Pereira,  {\it Homogenization in a thin domain with an oscillatory boundary. }
J. Math. Pures Appl. (9) 96 (2011), no. 1, 29-57.

\bibitem{ANaPe} Jos\'e M. Arrieta, Jean C. Nakasato, Marcone C. Pereira, {\it The p-Laplacian equation in thin domains: the unfolding approach. } J. Differential Equations 274 (2021), 1-34.

\bibitem{ANP} Jos\'e M. Arrieta,  Ariadne Nogueira, Marcone C.  Pereira,  {\it Semilinear elliptic equations in thin regions with terms concentrating on oscillatory boundaries. } 
Comput. Math. Appl. 77 (2019), no. 2, 536-554.
\bibitem{AVP}Jos\'e M. Arrieta, Manuel Villanueva-Pesqueira, {\it Elliptic and parabolic problems in thin domains with doubly weak oscillatory boundary. } Commun. Pure Appl. Anal. 19 (2020), no. 4, 1891-1914.
\bibitem{Ba93} Guy Barles, {\it Fully nonlinear Neumann type boundary conditions for second-order elliptic and parabolic equations. }
J. Differential Equations 106 (1993), no. 1, 90–106. 
\bibitem{BBI} Isabeau Birindelli, Ariela Briani, Hitoshi Ishii, {\it Test function approach to fully nonlinear equations in thin domains}, 
arXiv:2404.19577, April 2024. 
\bibitem{BFF} Andrea Braides, Irene Fonseca, Gilles Francfort, {\it 3D-2D asymptotic analysis for inhomogeneous thin films. } Indiana Univ. Math. J. 49 (2000), no. 4, 1367-1404. 
 \bibitem{CIL} Michael G. Crandall, Hitoshi  Ishii, Pierre-Louis Lions, {\it User's guide to viscosity solutions of second order partial differential equations. }Bull. Amer. Math. Soc. (N.S.) 27 (1992), no. 1, 1--67.
\bibitem{Ev} Lawerence C. Evans,  {\it The perturbed test function method for viscosity solutions of nonlinear PDE. } Proc. Roy. Soc. Edinburgh Sect. A 111 (1989), no. 3-4, 359-375. 
 \bibitem{GT} David Gilbarg, Neil S. Trudinger, Elliptic partial differential equations of second order,
{Classics Math.
Springer-Verlag, Berlin}, 2001, xiv+517 pp.
\bibitem{HR}Jack K. Hale, Genevi\`eve Raugel, {\it Reaction-diffusion equation on thin domains. }
J. Math. Pures Appl. (9) 71 (1992), no. 1, 33--95.
\bibitem{Is91} Hitoshi  Ishii,  {\it Fully nonlinear oblique derivative problems for nonlinear second-order elliptic PDE’s. }Duke Math. J., 62 (1991), no. 3,  633-661. 
\bibitem{P} Stefania Patrizi, {\it The Neumann problem for singular fully nonlinear operator. } 
J. Math. Pures Appl. (9) 90 (2008), no. 3, 286-311.
\bibitem{R} Genevi\`eve Raugel, {\it Dynamics of partial differential equations on thin domains. }
Lecture Notes in Math., 1609 Springer-Verlag, Berlin, 1995, 208-315.
\bibitem{PRPR}Prova
\end{thebibliography}
\end{document}